\documentclass[11pt]{article}

\usepackage[english]{babel}
\usepackage[utf8]{inputenc}
\usepackage[T1]{fontenc}
\usepackage{graphicx}
\usepackage{geometry}
\usepackage{bbm}
\usepackage{alltt}
\usepackage{amsthm,amsfonts}
\usepackage{amsmath}
\usepackage{amssymb}
\usepackage{mathabx}
\usepackage{titlesec}
\usepackage{mathtools}
\usepackage{empheq}
\usepackage{tikz}
\mathtoolsset{showonlyrefs=true}
\usepackage{xcolor}
\usepackage{enumitem}

\newcommand{\R}{\mathbb{R}}
\newcommand{\C}{\mathbb{C}}
\newcommand{\N}{\mathbb{N}}

\newcommand{\x}{\xi}
\newcommand{\fia}{\mathbbm{1}_{|\Phi-\alpha|<M}}
\newcommand{\fiM}{\mathbbm{1}_{|\Phi|<M}}

\newcommand{\jap}[1]{\left\langle #1 \right\rangle}

\numberwithin{equation}{section}
\newcommand{\be}{\begin{equation}}
\newcommand{\ee}{\end{equation}}

\theoremstyle{plain}
\newtheorem{thm}{Theorem}
\newtheorem*{thm*}{Theorem}
\newtheorem{prop}{Proposition}
\newtheorem{conj}{Conjecture}

\newtheorem{lem}{Lemma}

\theoremstyle{definition}

\theoremstyle{remark}
\newtheorem{nb}{Remark}

\makeatletter
\def\blfootnote{\xdef\@thefnmark{}\@footnotetext}
\makeatother

\title{Sharp local well-posedness and nonlinear smoothing for dispersive equations through frequency-restricted estimates}

\date{}%\nodate
\author{Simão Correia, Filipe Oliveira and Jorge Drumond Silva}

\begin{document}
\maketitle

\begin{abstract}
	We consider the problem of establishing nonlinear smoothing as a general feature of nonlinear dispersive equations, i.e. the improved regularity of the integral term in Duhamel's formula, with respect to the initial data and the corresponding regularity of the linear evolution, and how this property relates to local well-posedness. In a first step, we show how the problem generally reduces to the derivation of specific frequency-restricted estimates, which are multiplier estimates in the spatial frequency alone. Then, using a precise methodology, we prove these estimates for the specific cases of the modified Zakharov-Kuznetsov equation, the cubic and quintic nonlinear Schrödinger equation and the quartic Korteweg-de Vries equation.

\end{abstract}

\section{Introduction}

\subsection{Description of the problem and motivation}

Nonlinear smoothing denotes the phenomenon, typically observed in many nonlinear dispersive  equations, by which the
integral term in Duhamel's formula exhibits increased regularity when compared to the initial data of the Cauchy problem. This shows
that it is the linear term which is responsible for the persistence of the lower regularity of the flow, given by that initial data.

This property is closely related to \textit{dispersive blow-up}, where singularities occur in the flow of
nonlinear dispersive equations as a consequence solely of the linear component of their evolution, while the nonlinear contribution, given by the integral term in Duhamel's formula, remains controlled and comparatively well-behaved.

Dispersive blow-up was first conjectured to occur in \cite{BBM}, for the KdV equation, using a heuristic argument that raised the possibility of existence of a focusing effect of the linear evolution, as a consequence of specifically constructed initial data, due to the dispersion relation of the equation and its unbounded phase velocity. But a proof of such a result, in the context of  the generalized KdV equations, was only established by Bona and Saut \cite{BS1} about twenty years later. In this case, the proof of dispersive blow-up relies on a particular type of improved regularity of the nonlinear part of the evolution in which it remains bounded, while the linear part blows up.
This point of view, centered on establishing singularity formation exclusively from the linear component of the evolution of solutions to nonlinear dispersive equations, but where determining some form of nonlinear smoothing effect is required as a complementary feature towards the proof of the main result, has been the essence of
several published papers in the last few years (see e.g. \cite{BPSS,BS2,FLPS}). 

Very soon after \cite{BS1}, Linares and Scialom  in \cite{linaresscialom}  considered specifically the issue
of increase in Sobolev regularity of the nonlinear term in Duhamel's formula, proving, for the modified KdV equation, that this term is in $H^{s+1}(\R)$ for Cauchy data in $H^s(\R)$, for $s\geq 3/4$. Since then, this point of view 
of  associating the nonlinear smoothing phenomenon precisely  with the increase of regularity of the nonlinear component of the evolution, measured in  terms of Sobolev spatial derivatives, while  not necessarily trying to establish the occurrence of dispersive singularities, has become predominant.
Results along these lines were proved by Bourgain in \cite{bourgain} for the cubic, defocusing, nonlinear Schr\"odinger equation in $\R^2$ and by Keraani and Vargas (\cite{keraani}) for the general $L^2$-critical nonlinear Schr\"odinger equation in $\R^n$. M. Erdo\u{g}an, N. Tzirakis and their
collaborators have extensively studied this property for several classes of equations and settings, in one spatial dimension (see \cite{tzirakis7, tzirakis1, tzirakis2} and references therein) by using a combination of normal form reduction together with Bourgain's Fourier restriction norm method ($X^{s,b}$ spaces). In the periodic and almost-periodic settings, this approach has also been sucessfully applied in \cite{IMOS} and \cite{CKV}.

In a previous paper \cite{CS}, two of the current authors used the method of the infinite iteration of normal form reductions (INFR), developed in \cite{ko} and \cite{koy} to study unconditional well-posedness, to obtain nonlinear smoothing properties for the KdV, modified KdV, cubic and derivative nonlinear  Schr\"odinger equations in $\R$, as well as for the modified Zakharov-Kuznetsov equation in $\R^2$. 

Although in that article the goal was to use the INFR as a very general method for obtaining nonlinear smoothing for some  dispersive partial differential equations, we now intend to further extend the idea that the nonlinear smoothing effect is actually an intrinsic and general
property of  nonlinear dispersive equations. With that purpose, we will use here multilinear estimates in Bourgain (or Fourier restriction norm) spaces to estimate and prove the gain of regularity of the nonlinear terms of the equations. This highlights the strong relationship with the local well-posedness regularity, for which the same methods have been used since the 1990's for optimal results by Picard iteration, as for example in \cite{hadac, KPV6, kino}.

To briefly summarize the basic ideas, let us consider  a generic dispersive PDE 
\begin{equation}\label{eq:geral}
	u_t - i L(D)u = N(u), \quad u\big|_{t=0}=u_0\in H^s(\R^d),
\end{equation}
where $L(D)$, $D=\partial_x/i$, is a linear differential (or pseudo-differential) operator in the spatial variables  given by a real Fourier symbol $L(\xi)$, and $N(u)$ is the nonlinear term, typically a combination of derivatives and powers of $u$ and its conjugate $\bar{u}$. If $(G(t))_{t\in\R}$ denotes the associated linear group, given by 
$G(t)u_0=e^{itL(D)}u_0=\mathcal{F}^{-1}\left( e^{itL(\xi)}\widehat{u_0}(\xi)\right) $, then the solution
to this initial value problem is given in integral form  by the Duhamel formula:
\begin{equation}\label{fullduhamel}
	u(t)=G(t)u_0 +\int_0^t G(t-t')N(u)(t')dt'.
\end{equation}

Local well-posedness by a Picard iteration scheme requires the use of an appropriate norm in order to obtain the contraction that yields the fixed
point of \eqref{fullduhamel}. By using Bourgain's $X^{s,b}$ spaces adapted to the linear evolution, with norm given by
\begin{equation}\label{bourgain:norm}
\|u\|_{X^{s,b}}=\|u\|_{X^{s,b}_{\tau=L(\xi)}}:=\|\langle\xi \rangle ^s \langle \tau -L(\xi) \rangle ^b \mathcal{F}_{t,x}(u)(\tau,\xi)\|_{L^2_{\tau,\xi}},
\end{equation}
where $\langle \cdot \rangle = (1+|\cdot|^2)^{1/2}$ and ${F}_{t,x}$ denotes the Fourier transform in both time and space variables, the essential point of the proof 
boils down to establishing a multilinear estimate  of the type
\begin{equation}\label{bourgain:mult}
	\|N(u)\|_{X^{s,b-1}} \leq C \|u\|^k_{X^{s,b}},
\end{equation}
for $b>1/2$, where $k$ is the power of the nonlinearity. The reason for the loss of one time derivative $b-1$ on the left-hand side of the estimate is the result of moving the norm from outside to inside the time integral, which then requires the multilinear estimate to pull the time regularity back up to $b$ on the right hand side. The regularity in time, $b$, must be larger than $1/2$ so that 
$X^{s,b} \subset C_t H^s_x$ in order for the solution obtained, i.e. the fixed point of the Picard iteration in the space
$X^{s,b}$, to actually be a time continuous flow in spatial $H^s$.

%We assume the existence of a nice local well-posedness theory: 
%\begin{assumpt}
%	Given $u_0\in H^{s_0}(\real^d)$, there exist $T=T(\|u_0\|_{H^{s_0}})$ and a solution $u\in C([0,T],H^{s_0}(\real^d))$ of \eqref{eq:geral}, depending continuously on the initial data. Moreover, if $u_0\in H^s(\real^d)$, $s\ge {s_0}$, then $u\in C([0,T], H^s(\real^d))$.
%\end{assumpt}

We are interested in the nonlinear smoothing property of \eqref{eq:geral}, which can be stated as follows: 
given initial data $u_0\in H^s(\R^d)$, the difference between the full nonlinear evolution
of the initial value problem \eqref{eq:geral}, represented by \eqref{fullduhamel}, and the corresponding linear evolution, for the same initial data, satisfies
$$
\|u(t)-G(t)u_0\|_{H^{s+\epsilon}}=\left\|\int_0^t G(t-t')N(u)(t')dt'\right\|_{H^{s+\epsilon}}\le C(t,\|u_0\|_{H^s}).
$$

As above, in terms of Bourgain spaces, this follows from a multilinear estimate analogous to \eqref{bourgain:mult}, where the smoothing
effect is now reflected as
\begin{equation}\label{bourgain:smooth}
	\|N(u)\|_{X^{s+\epsilon,b-1}} \leq C \|u\|^k_{X^{s,b}},
\end{equation}
again for some $b>1/2$: the injection $X^{s+\epsilon,b} \subset C_t H^{s+\epsilon}_x$ implies that the $H^{s+\epsilon}$ norm of the Duhamel
integral is bounded by its $X^{s+\epsilon,b}$ norm which, in turn, is controlled by $\|N(u)\|_{X^{s+\epsilon,b-1}}$ from the standard
linear estimates with Fourier restriction norms.

%Observe that the multilinear estimate \eqref{bourgain:smooth}, with improved regularity, is actually a stronger estimate than \eqref{bourgain:mult} required for the proof of local well-posedness, as one can obtain the latter simply from the case $\epsilon=0$. The connection between these two multilinear estimates, as well as their methods of proof, exposes the intimate relationship between the local well-posedness properties of nonlinear dispersive equations, in particular the minimum Sobolev regularity for it to hold (at least through a Picard iteration scheme) and nonlinear smoothing properties.
%
%Having reduced both local well-posedness and nonlinear smoothing problems to the derivation of multilinear estimates in Bourgain spaces (in $\tau,\xi$ variables), we will show that a further reduction can be made, replacing the latter with \textit{frequency-restricted} estimates, which are multiplier estimates only in the spatial frequency $\xi$.
%This  approach of proving multilinear estimates in Bourgain spaces, implying in particular local well-posedness and nonlinear smoothing, is the central subject of this paper. 
Observe that the multilinear estimate \eqref{bourgain:smooth}, with improved regularity, is actually a stronger estimate than \eqref{bourgain:mult} required for the proof of local well-posedness, as one can obtain the latter simply from the case $\epsilon=0$. The connection between these two multilinear estimates, as well as their methods of proof, exposes the intimate relationship between the local well-posedness properties of nonlinear dispersive equations, in particular the minimum Sobolev regularity for it to hold (at least through a Picard iteration scheme) and nonlinear smoothing properties.

In this work, we will show that the derivation of multilinear estimates in Bourgain spaces (in $\tau,\xi$ variables), like \eqref{bourgain:mult} and \eqref{bourgain:smooth}, can further be reduced to frequency-restricted estimates, which are multiplier estimates in the spatial frequency $\xi$ only.
This approach to proving multilinear estimates  through frequency-restricted estimates, yielding new nonlinear smoothing results for dispersive equations and from which one can recover their known local well-posedness results in the literature, is the central subject of this paper. 
%The similarity between these two multilinear estimates, as well as their methods of proof, exposes the intimate connection between the local well-posedness properties of  nonlinear dispersive equations, in particular the minimum Sobolev regularity for it to hold (at least through a Picard iteration scheme) and nonlinear smoothing properties.
%
%This approach to proving nonlinear smoothing of dispersive equations, through multilinear estimates in Bourgain spaces, akin to the ones required to prove local well-posedness, is the central subject of this paper.

\subsection{Statement of the main results}

Following the general method described above, we will prove results for five specific cases of dispersive partial differential equations.

We start with the Zakharov-Kuznetsov family of equations
\begin{equation}\label{ZK}\tag{gZK}
\partial_t u + \partial_x \Delta_{x,{\bf y}} u = \partial_x u^k,\qquad u(0,x,{\bf y})=u_0(x,{\bf y}) \in H^s(\R^d),
\end{equation}
for $k \in \N$, $(x, {\bf y}) \in \R \times \R^{d-1}$ and $u:\R^{d+1}\to \R$, which are higher dimensional generalizations of
the Korteweg-de Vries (KdV) equation. Following the same terminology as for the KdV equation, the standard Zakharov-Kuznetsov (ZK)
equation corresponds to the quadratic nonlinearity, i.e. $k=2$, whereas the case $k=3$ and $k\geq 4$ are called, respectively, the modified (mZK) and generalized (gZK)
Zakharov-Kuznetsov  equations. In this paper we will only be studying the mZK equation, in two and higher dimensions, so that the nonlinearity will always be  $\partial_x u^3.$

We begin with the particular case $d=2$.

\begin{thm}\label{thm:2dmzk}
	Fix $s>1/4$. Then the $H^s$-flow associated to the modified Zakharov-Kuznetsov equation on $\R^2$ exhibits a nonlinear smoothing effect of any order $\epsilon<\min\{2s-1/2,1\}$.
\end{thm}
The modified Zakharov-Kuznetsov equation on $\R^2$ was proved to be locally well-posed in $H^s(\R^2)$, for $s>1/4$, in \cite{ribaudvento},
and for the endpoint $s=1/4$ in \cite{kinomzk}. Observe that the scaling critical regularity for this equation in $\R^2$ is $s_c=0$, but 
the Picard iteration method fails below $s=1/4$ as it was also proved in \cite{kinomzk} that the data to solution map then fails to be
$C^3$ and therefore is no longer analytic. Concerning the nonlinear smoothing properties of this equation, it was proved in \cite{FLPS} that
there is a gain of one derivative in the nonlinear flow for initial data in $H^s(\R^2)$, with $s\geq 1$. The current result improves
this, as the gain of one derivative is now shown to occur for $s\geq 3/4$ while also establishing that the nonlinear smoothing effect occurs starting from the optimal local well-posedness regularity that can be achieved through Picard iteration.

\bigskip

%Unlike in two dimensions, for $d\geq 3$ there is no linear change of coordinates that yields a symmetrization of the
%dispersion's third order derivative in the equation, as above, so a proof of nonlinear smoothing has to deal with
%the equation in its original form \eqref{ZK}. Nevertheless, we still obtain the following result.
%

For dimension greater or equal than three, we are able to prove nonlinear smoothing down to the scaling regularity:

\begin{thm}\label{thm:3dmzk}
	Take $d\ge 3$ and $s>d/2-1$. Then the $H^s$-flow associated to the modified Zakharov-Kuznetsov equation on $\R^d$ exhibits a nonlinear smoothing effect of any order $\epsilon<\min\{2s-d+2,1\}$.
\end{thm}
There are few works studying higher dimensional versions of the mZK equation. Local well-posedness in three
dimensions was proved in the full subcritical range $s>1/2$ in \cite{axelmzk}, while small data global
well-posedness in the scaling critical regularity $s_c=d/2-1$ is established in \cite{kinomzk}. No nonlinear
smoothing results were previously known for the mZK equation, in dimensions three or higher.

\bigskip

The nonlinear Schrödinger  equation, arguably the most studied of all the nonlinear dispersive partial differential equations, is given by
\begin{equation}\label{Schr}\tag{NLS}
	i\partial_t u + \Delta u = |u|^\alpha u,\qquad u(0,x)=u_0(x) \in H^s(\R^d),
\end{equation}
for $\alpha \in \R^+$ and $u:\R^{d+1}\to \C$. Although general, non integer, powers of the nonlinearity are usually considered for the \ref{Schr} equation, our method can only be applied to integer powers as we use the Fourier transform to turn the products into convolutions. Specifically, in order to write the absolute value of the
complex solution in terms of the product of $u$ with its conjugate $\bar u$, we will only consider here even integers $\alpha=2n$, $n \in \N$, and we prove the following nonlinear smoothing results for the cubic ($n=1$) and quintic ($n=2$) nonlinearities in any dimension.
\begin{thm}\label{thm:NLScubic}
	For $d\ge 2$ and $s>d/2-1$, the $H^s$-flow associated to the cubic nonlinear Schrödinger equation on $\R^d$ exhibits a nonlinear smoothing effect of any order $\epsilon<\min\{2s-d+2,1\}$.
\end{thm}

\begin{thm}\label{thm:NLSquintic}
	For $d\ge 2$ and $s>(d-1)/2$, the $H^s$-flow associated to the quintic nonlinear Schrödinger equation on $\R^d$ exhibits a nonlinear smoothing effect of any order $\epsilon<\min\{4s+2-2d,1\}$.
\end{thm}
The local well-posedness theory for the NLS has been known for many years, having been proved in \cite{caze} down to the scaling critical regularity
$s_c=d/2-1/n$. Nonlinear smoothing effects have also been studied in dimensions higher than one, in \cite{BPSS} and \cite{bourgain}, with the former showing a gain of almost a full derivative in the nonlinear flow for initial data in 
$H^s(\R^d)$, with $s>d/2-1/(4n)$ and powers of the nonlinearity cubic or higher, which is our case. Our results here, however,
show that a gain of a full derivative occurs for $s>d/2-1/2$ and $s>d/2-1/4$, for the cubic and quintic cases respectively, improving
the result in \cite{BPSS}. 

\bigskip

Our last result concerns the quartic, generalized Korteweg-de Vries equation
\begin{equation}\label{4kdv}\tag{4-KdV}
	\partial_t u + \partial_x^3  u = \partial_x u^4,\qquad u(0,x)=u_0(x) \in H^s(\R),
\end{equation}
with $u:\R^2 \to \R$, for which we establish the following nonlinear smoothing result:
\begin{thm}\label{thm:KdV}
	Fix $s>-1/6$. Then the $H^s$-flow associated to the quartic Korteweg-de Vries equation on $\mathbb{R}$ exhibits a nonlinear smoothing effect of any order $\epsilon<\min\{ 3s+1/2, s+1/2, 1\}$.
\end{thm}
Local well-posedness for the \ref{4kdv} equation was proved in \cite{axel}, for $s>-1/6$, which is precisely the scaling critical
regularity for this equation. To the best of our knowledge, no results of nonlinear smoothing properties for this equation were known before. 

%
%\begin{nb}
%The nonlinear smoothing property can be seen as an improvement on the standard local existence theory. In particular, taking $\epsilon=0$ in the various proofs of the present work, one recovers the local well-posedness results proven in the literature. 
%\end{nb}
\vskip15pt

The results described above, together with \cite{CS, tzirakis7, tzirakis1, tzirakis2, IMT}, among others, indicate that the nonlinear smoothing effect is a general feature of dispersive equations. The diversity of these situations has lead us to formulate the following conjecture: 

\begin{conj}
Given a nonlinear dispersive equation with dispersion of order $l$ and a nonlinearity of order $k\ge 3$ with a loss of $m$ derivatives, let $s_0$ be such that, for $s>s_0$, the $H^s$-flow exists and is analytic. Then the flow exhibits a nonlinear smoothing effect of order $\epsilon<\min\{(k-1)(s-s_0),l-m-1\}$.
\end{conj}
\begin{nb}
	It is well known, for example with the well-studied example of the KdV equation, that the optimal regularity at which local well-posedness
	can be obtained by the Picard iteration scheme \cite{KPV6, kishi} can be significantly higher than the expected
	one from heuristic scaling arguments due to breakdown of the \textit{analytic} dependence of the data to flow map and a corresponding
	type of ill-posedness below that Picard threshold (\cite{CCT1, KPV4}). However, one can see in \cite{HKV, KiVi, MST}, that the range of regularity for  mere \textit{existence} of the flow can be extended below that for which it is \textit{analytic}, with more severe
	ill-posedness conditions at lower levels of regularity as for example in \cite{mol}, where it is proved that the data to flow map becomes discontinuous. It is an open problem whether nonlinear smoothing is possible for non-analytic flows, in that intermediate range.
\end{nb}
\begin{nb}
When $k=2$, the conjecture is false. Indeed, Isaza, Mejía and Tzevtkov \cite{IMT} have proven that the Korteweg-de Vries equation does not exhibit nonlinear smoothing in terms of $H^s$-spaces (see also \cite{CS, LPS} for results in weighted spaces). As we will see later (Remark \ref{nota:k=2}), the quadratic nonlinearities pose severe obstacles to our strategy, in agreement with the cited literature.
\end{nb}
\begin{nb}
We do not claim that the bound on $\epsilon$ is optimal. On the contrary, we expect that specific nonlinearities may avoid the problematic resonances and give rise to improved gains in regularity (see, for example, the gauge-transformed derivative nonlinear Schrödinger equation in \cite{CS}).
\end{nb}

\begin{nb}
	For the quartic Korteweg-de Vries equation, there is a gap betweem Theorem 5 and Conjecture 1. The bound $\epsilon<s+1/2$ is related to some specific frequency interactions in the corresponding multilinear estimate. However, while the multilinear estimate does fail, the nonlinear estimate might still hold.
\end{nb}

\subsection{Methodology of the proofs}

As outlined in \eqref{bourgain:smooth},  Theorems 1-5 will follow from the corresponding multilinear estimates in Bourgain spaces $X^{s,b}=X^{s,b}_{\tau=\phi(\xi)}$. Write a general  nonlinearity $N(u)$ in  multilinear form, as $N(u)=N[u, \ldots, u]$ where\footnote{To make the presentation simpler in this introduction, we are assuming here that conjugates of $u$ do not occur in the nonlinear term $N(u)$. However,
Bourgain spaces are not invariant under conjugation, and the appearance of terms in $\bar{u}$ in the nonlinearity needs to be carefully considered, which we will do once the
technical details of the specific proofs are presented in Section \ref{multsection}.}

\begin{equation}\label{multilinearform}
	\widehat{N[u_1,\ldots,u_k]}(\xi)=\int_{\xi=\xi_1+\cdots+\xi_k}m(\Xi)\widehat{u_1}(\xi_1) \cdots \widehat{u_k}(\xi_1),
\end{equation}
with $k$ the order of the nonlinearity, $\Xi=(\xi,\xi_1,\ldots,\xi_k)$ and $m(\Xi)$ a Fourier multiplier that represents the
possible derivatives ocurring in $N(u)$.
It suffices to prove that, for $b=(1/2)^+$, $s'=s+\epsilon$ and $b'=(b-1)^+$,
\begin{equation}\label{eq:multibourgain}
	\|N[u_1,\dots,u_k]\|_{X^{s',b'}} \lesssim \prod_{j=1}^k\|u_j\|_{X^{s,b}}.
\end{equation}
Using duality, \eqref{eq:multibourgain} becomes, for all $u\in X^{-s',-b'}$ and all $u_j\in X^{s,b}$,
$$\Big|\int_{(x,t)}N[u_1,\dots,u_k]\bar{u} dxdt\Big| \lesssim \|u\|_{X^{-s',-b'}}\prod_{j=1}^k\|u_j\|_{X^{s,b}}.$$
Applying Plancherel, noticing that $\widehat{N[u_1,\dots,u_k]}=m(\Xi)\hat{u_1}*\dots*\hat{u_k}$ and putting 
$$\hat u=\langle\tau-\phi(\xi)\rangle^{b'}\langle\xi\rangle^{s'}\overline{v},\quad \hat u_j=\langle\tau_j-\phi(\xi_j)\rangle^{-b}\langle\xi_j\rangle^{-s}v_j,$$
we obtain, for all $v_j,v\in L^2$,
\begin{equation}\label{eq:duality}
	\left|\int_{(\tau,\xi)=\sum_{j=1}^k (\tau_j,\xi_j)}\frac{m(\Xi)\jap{\xi}^{s'}\jap{\tau-\phi(\xi)}^{b'}}{\prod_{j=1}^k\jap{\xi_j}^s\jap{\tau_j-\phi(\xi_j)}^b}v\prod_{j=1}^{k}v_jd\xi_1\dots d \xi_{k}d\tau_1\dots d\tau_k\right|\lesssim \|v\|_{L^2}\prod_{j=1}^k\|v_j\|_{L^2}.
\end{equation}
The structured derivation of multilinear estimates of the form \eqref{eq:duality} has been the subject of intense research in the last twenty years. With this general
perspective to proving multilinear estimates, one of the main references in the field is the work of Tao \cite{tao_weightedL2}, where a generic procedure to tackle this problem is presented, either through $L^2$ estimates of the multiplier or through a more sophisticated version of Schur's lemma. All these arguments ultimately  rely on a very specific knowledge of the phase function (in particular, on the structure of its zero set, called the \textit{resonant} set, and of its critical point set, the \textit{coherent} set). However, as the resonant set can become quite intricate for large dimensions and/or higher-order nonlinearities, the method must then be adapted \textit{ad hoc} to each specific case. In this respect, Tao even points out in \cite[pg. 16]{tao_weightedL2},  ``$X^{s,b}$ \textit{norms are better at controlling the effects of resonance,
whereas physical space norms (such as mixed Lebesgue norms) are better at controlling the effect of coherence. (...)  It is not clear at present what the best way to combine
these two types of norms is, or whether completely new norms are needed.}'' Despite this observation, we are able to prove our 
results exclusively within the framework of Bourgain spaces, which might be due to the fact that the Strichartz and mixed Lebesgue norm estimates
are just another facet of the dispersive effects, already encoded in them.

The analysis of dispersive equations based on resonant and coherent frequencies can also be seen in the \textit{space-time resonances} method of Germain, Masmoudi and Shatah, first introduced in \cite{GMS}. There, time resonances correspond to Tao's resonant frequencies and space resonances correspond to coherence. Their arguments work particularly well for quadratic nonlinearities, where, once again, one can understand precisely the structure of the problematic frequencies. Generally speaking, for nonresonances in time, the improvement in the estimates comes from a normal form reduction (more specifically, an integration by parts in time). On the other hand, for nonresonances in space, the authors use vector-fields, as introduced by Klainerman in \cite{klainerman} (which is related to an integration by parts in frequency). Notice the difference between the two arguments: while integration by parts in time uses the equation itself (something which is intrinsic to the evolution problem), the integration by parts in frequency requires extra information from the solution. 

In the present work, we propose a systematic method to derive multilinear estimates, in which one can see the ingredients present in \cite{GMS} or \cite{tao_weightedL2} arising in a natural way. The first step is to reduce the problem to an estimate on the corresponding multiplier, in the spirit of \cite{tao_weightedL2}, using either Cauchy-Schwarz or an analogue of the Schur's lemma. The idea is to reduce \eqref{eq:duality} to an estimate \textit{only in the spatial frequencies}, while keeping the oscillatory information. Let 
$$\mathcal{M}:=\frac{m(\Xi)\jap{\xi}^{s'}}{\prod_{j=1}^k\jap{\xi_j}^s}
$$ 
and $\Phi$ the (total) phase function (see \eqref{eq:phase}). In Lemmas \ref{lem:CS} and \ref{lem:interpol}, we prove that the multilinear estimate holds if either:
\begin{itemize}
	\item (Cauchy-Schwarz) there exists $\beta<1/2$ such that 
\begin{equation}\label{eq:cauchy}
		\sup_{\alpha,\xi} \left(\int_{\xi_1+\dots +\xi_k=\xi} \mathcal{M}^2\fia d\xi_1\dots d\xi_{k-1} \right)^{\frac{1}{2}} \lesssim M^\beta ,\quad M>1.
\end{equation}
	\item (Schur's test) for some $1\le j<k$ and $\beta<1$,
\begin{multline}\label{eq:schur}
		\sup_{\alpha, \xi,\xi_1,\dots,\xi_{j}} \left(\int_{\xi_1+\dots +\xi_k=\xi} \mathcal{M}\fia d\xi_{j+1}\dots d\xi_{k-1}\right)\\  + \sup_{\alpha, \xi_{j+1},\dots,\xi_k} \left(\int_{\xi_1+\dots +\xi_k=\xi} \mathcal{M}\fia d\xi_1\dots d\xi_{j} \right) \lesssim M^\beta,\quad  M>1.
\end{multline}
\end{itemize} 
These \textit{frequency-restricted} estimates can be further generalized: for example, one can consider different multipliers in \eqref{eq:schur}, as long as their geometric mean is $\mathcal{M}$ (see Section 2).

\begin{nb}
	Estimates \eqref{eq:cauchy} and \eqref{eq:schur} can be viewed as \textit{sublevel estimates} for the phase function. The parallel between oscillatory integrals and sublevel estimates (which is basically a multidimensional generalization of Van der Corput's lemma) has been studied in \cite{CCW} (see also \cite{Ruixiang}). Even though the passage from a multilinear Bourgain estimate to a frequency-restricted estimate can be morally viewed as a manifestation of this parallel, it is not clear how to include frequency-restricted estimates in the existing literature of sublevel estimates. 
\end{nb}

The Cauchy-Schwarz frequency-restricted estimate has already appeared in the derivation of multilinear estimates (see, for example, \cite[Lemma 3.9]{tao_weightedL2}) and in the context of the infinite normal form reduction (\cite{koy, mosincatyoon}). The difficulty with the Cauchy-Schwarz argument is that it imposes $\beta<1/2$, while \textit{the scaling-critical regularity corresponds to} $\beta=1$. In several cases, such as the cubic NLS or the modified KdV on $\R$, this is already enough to reach sharp local well-posedness and optimal nonlinear smoothing results \cite{CS}. However, for the equations considered in this work, \eqref{eq:cauchy} is far from being optimal and more refined ideas are necessary.

Let us motivate the frequency-restricted estimate \eqref{eq:schur}. In order to reach the scaling-critical case $\beta=1$, we must reduce \eqref{eq:duality} to an estimate on the multiplier $\mathcal{M}$  losing almost nothing in the process (this is obviously not what happens using Cauchy-Schwarz). Our idea can be traced back to Schur's test, which gives conditions for which
\begin{equation}\label{eq:schurlemma}
	\left|\int K(x,y)f(x)g(y)dxdy \right|\lesssim_K \|f\|_{L^2}\|g\|_{L^2}. 
\end{equation}
If one applies Cauchy-Schwarz, we see that the estimate holds when $K$ is a Hilbert-Schmidt kernel
$$
\left(\int |K(x,y)|^2 dxdy\right)^{\frac{1}{2}}<+\infty.
$$
which is the analogous condition to \eqref{eq:cauchy}. On the other hand, Schur's lemma states that \eqref{eq:schurlemma} holds under the conditions
$$
\sup_x \int |K(x,y)| dy < +\infty\qquad\textrm{ and }\qquad \sup_y \int |K(x,y)| dx < +\infty.
$$
There are several proofs for Schur's lemma. One of them is based on a simple interpolation argument between $L^1$ and $L^\infty$. Indeed, if $f\in L^1$ and $g\in L^\infty$,
\begin{align}
	\left|\int K(x,y)f(x)g(y)dxdy \right|&\lesssim \left(\int \left(\int|K(x,y)|dy\right)|f(x)|dx \right)\|g\|_{L^\infty}\\&\lesssim \sup_x\left( \int|K(x,y)|dy\right)\|f\|_{L^1}\|g\|_{L^\infty}.\label{eq:schurproof}
\end{align}
By symmetry, an analogous estimate holds for $f\in L^\infty$ and $g\in L^1$, and the result follows by interpolation. The reason why we present the proof of such a well-known result is to highlight the \textit{almost optimality} of \eqref{eq:schurproof}, passing from a bilinear estimate to a multiplier estimate. In Section 2, we apply similar arguments to reduce the multilinear estimate \eqref{eq:duality} to the frequency-restricted estimate \eqref{eq:schur}.

After having reduced our problem, let us now explain the guiding principles that we apply in this work to prove the necessary frequency-restricted estimates. To fix some ideas, let us assume that $d=1$, $k=3$, $\Phi$ is homogeneous of degree $l$ and that $|\xi|\ge |\xi_1|,|\xi_2|, |\xi_3|$. Considering  $j=1$, we exemplify how the first integral in the left-hand-side of \eqref{eq:schur},
$$
\sup_{\alpha,\xi,\xi_1} \int_{\xi_1+\xi_2+\xi_3=\xi} \mathcal{M}\fia d\xi_2,
$$
can be  bounded. We split the study into two cases:
\begin{itemize}
	\item $|\partial_{\xi_2} \Phi|\gtrsim |\xi|^{l-1}$. Then one may perform the change of variables $\xi_2\mapsto \Phi$:
\begin{align*}
		\sup_{\alpha,\xi,\xi_1} \int \mathcal{M}(\xi,\xi_1,\xi_2)\fia d\xi_2&\lesssim \sup_{\alpha,\xi,\xi_1} \int \left(\sup_{\xi_2} \mathcal{M}(\xi,\xi_1,\xi_2)\right)\fia \frac{1}{|\xi|^{l-1}}d\Phi\\&\leq 2\left(\sup_{\alpha,\xi,\xi_1,\xi_2}\frac{\mathcal{M}(\xi,\xi_1,\xi_2)}{|\xi|^{l-1}}\right)M.
\end{align*}
\item $|\partial_{\xi_2}\Phi|\ll |\xi|^{l-1}$. Writing
$$
\Phi(\xi,\xi_1,\xi_2,\xi_3) = \xi^lP(p_1,p_2,p_3)=\xi^lP(\mathbf{p}),\quad p_j=\frac{\xi_j}{\xi},\ j=1,2,3,
$$
where $P(p_1,p_2,p_3)=\Phi(1,p_1,p_2,p_3)$, the stationarity condition will usually mean we are close to some explicit stationary point $\mathbf{p}^0$. When dealing with this case, we will use two general facts. First, if the components of $\mathbf{p}^0$ are nonzero, then $\xi_j$ will be comparable to $\xi$, leading to a large simplification in $\mathcal{M}$. Second, as $|\mathbf{p}-\mathbf{p}^0|\ll1$, we can apply Morse's lemma to extract the precise nature of the restriction $|\Phi-\alpha|<M$.
\end{itemize}
In Section 3. these principles can be observed (with various degrees of complexity) in the proof of \textit{every single frequency-restricted estimate}. Moreover, these ideas were already present in \cite{CS} and \cite{C-BO}, where the modified KdV, the 1D cubic NLS, the derivative NLS and the Benjamin-Ono equation were analyzed. In conjunction with the present work, we see the successful application of this methodology to \textit{nine prototypical dispersive equations}, with various spatial dimensions, dispersions and nonlinearities.
 We invite the reader to compare the complexity between the proofs in the literature and our arguments.

\begin{nb}\label{nota:espaco}
	Another feature of our approach is that it helps understanding the reason why lower order nonlinearities and/or lower dimensions can pose obstacles to the local well-posedness theory up to the scaling-critical exponent. Speaking loosely, as the dimension of the total frequency space $(\R^d)^k\simeq \R^{dk}$ decreases, the less \textit{room} we have to perform integrations and apply Morse's lemma. A symptom of this lack of room  can be seen in the sublevel estimates for quadratic polynomials in one and two dimensions (Lemma \ref{lem:quadraticas}),
	\begin{equation}
		 \int \mathbbm{1}_{|p^2-\alpha|<M}dp\lesssim M^{1/2}\quad\mbox{and}\quad \int_{|p|,|q|<N} \mathbbm{1}_{|p^2\pm q^2-\alpha|<M}dpdq\lesssim M^{1^-}N^{0^+},\quad M>1.
	\end{equation}
While, in two dimensions, one is able to reach $\beta=1$, the one-dimensional case cannot go above $\beta=1/2$, resulting in a considerable gap from the scaling-critical regularity.
\end{nb}

\begin{nb}\label{nota:k=2}
	When the nonlinearity is quadratic, the lack of room is even more dire.  This impediment can be seen in the frequency restricted estimate \eqref{eq:schur}. Indeed, when $k=2$ and $j=1$ we get
	$$
	\sup_{\alpha,\xi,\xi_1} \int_{\xi_2=\xi-\xi_1} \mathcal{M}\fia = 	\sup_{\alpha,\xi,\xi_1}\mathcal{M}(\xi,\xi_1,\xi-\xi_1)\fia = \sup_{ \xi,\xi_1}\mathcal{M}(\xi,\xi_1,\xi-\xi_1).
	$$
	The lack of integration completely kills the oscillatory information encoded in $\fia$ and the resulting estimate will not be optimal. This observation agrees with the literature: it is well-known that \eqref{bourgain:mult} may fail and slight modifications of the classic Bourgain spaces may be necessary (as, for example, in \cite{Tao-Bejenaru}, \cite{hadac} or \cite{kino}).
\end{nb}

\vskip10pt
The rest of this work is organized as follows. In Section 2, we reduce the  mulitlinear estimate \eqref{eq:multibourgain} to various frequency-restricted estimates. Then, in Section 3, we prove the frequency-restricted estimates for the modified Zakharov-Kuznetsov equation in two and higher dimensions (Subsection 3.1), the cubic and quintic NLS (Subsection 3.2) and finally the quartic Korteweg-de Vries (Subsection 3.3).

\vskip10pt
\textbf{Notation.} The Fourier transform of $u$ in the spatial variable will be denoted by $\hat{u}$ and the space-time Fourier transform by $\mathcal{F}_{t,x}u$. Given $x\in \R$, $x^+$ (resp. $x^-$) denotes any number sufficiently close to $x$ which is greater (resp. smaller) then $x$. In two  (resp. three) dimensions, we will write the frequency $\xi$ as $(x,y)$ (resp. $(x,y,z)$). Given $a,b\ge 0$, $a\lesssim b$ means there exists a (universal) constant $C>0$ such that $a\le Cb$. If $a\lesssim b \lesssim a$, we write $a\sim b$. If there exists a small $\delta>0$ such that $a\le \delta b$, we say that $a\ll b$. Finally, if $\xi_1,\xi_2\in \R^d$ satisfy $|\xi_1-\xi_2|\ll \max\{|\xi_1|,|\xi_2|\}$, we say that $\xi_1\simeq \xi_2$.
\vskip10pt 
\noindent \textbf{Acknowledgements.} S. Correia and J. Silva were partially supported by Funda\c{c}\~ao para a Ci\^encia e Tecnologia, through CAMGSD, IST-ID
(projects UIDB/04459/2020 and UIDP/04459/2020). F. Oliveira was partially supported by Funda\c{c}\~ao para a Ci\^encia e Tecnologia, through CEMAPRE (project UIDB/05069/2020). The authors were also supported through the FCT project NoDES (PTDC/
MAT-PUR/1788/2020).

%\begin{lem}
%	For $N, M>0$ and  $\alpha\in \R$ fixed, 
%	\begin{equation}\label{eq:polares}
%		\int_{|p|,|q|<N} \mathbbm{1}_{|p^2+q^2-\alpha|<M}dpdq\lesssim M^{1^-}N^{0^+}
%	\end{equation}
%	and
%	\begin{equation}\label{eq:hiperb}
%		\int_{|p|,|q|<N} \mathbbm{1}_{|p^2-q^2-\alpha|<M}dpdq\lesssim M^{1^-}N^{0^+}
%	\end{equation}
%\end{lem}
%\begin{proof}
%	For \eqref{eq:polares}, it suffices to use Hölder:
%	\begin{align*}
%		\int_{|p|,|q|<N} \mathbbm{1}_{|p^2+q^2-\alpha|<M}dpdq\lesssim \left(\int_{|p|,|q|<N}dp_1dq_1\right)^{0^+}\left(\int \mathbbm{1}_{|p^2+q^2-\alpha|<M}dpdq\right)^{1^-} M^{1^-}N^{0^+}.
%	\end{align*}
%	For \eqref{eq:hiperb}, 
%	\begin{align*}
%		\int_{|p|,|q|<N} \mathbbm{1}_{|p^2-q^2-\alpha|<M}dpdq&\lesssim 	\int_{|p_1|,|q_1|<2N} \mathbbm{1}_{|p_1q_1-\alpha|<M}dp_1dq_1\\&\lesssim 	\int_{|q_1|<2N}\left(\int_{|p_1|<2N}dp_1\right)^{0^+} \left(\int\mathbbm{1}_{|p_1q_1-\alpha|<M}dp_1\right)^{1^-}dq_1 \\&\lesssim N^{0^+}\int_{|q_1|<2N} \frac{M^{1^-}}{|q_1|^{1^-}}dq_1\lesssim N^{0^+}M^{1^-}
%	\end{align*}
%\end{proof}

\section{Reduction of multilinear estimates to frequency-restricted estimates}\label{multsection}

In this section, we consider  a general $k$-multilinear term, now including the possible appearance of conjugates, of the form
\begin{align*}
\widehat{N[u_1,\dots,u_k]}(\xi)&=\int_{\xi=\xi_1+\dots+\xi_k}m(\Xi)\left(\prod_{j=1}^{k_0}\hat{u}_j(\xi_j)\right)\left(\prod_{j=k_0+1}^{k}\widehat{\overline{u_j}}(\xi_j) \right)\\&=\int_{\xi=\xi_1+\dots+\xi_k}m(\Xi)\left(\prod_{j=1}^{k_0}\hat{u}_j(\xi_j)\right)\left(\prod_{j=k_0+1}^{k}\overline{\widehat{{u_j}}}(-\xi_j) \right)\\&=\int_{\xi=\xi_1+\dots+\xi_{k_0}-\xi_{k_0+1}\dots-\xi_k}m(\xi_1,\dots,\xi_k)\left(\prod_{j=1}^{k_0}\hat{u}_j(\xi_j)\right)\left(\prod_{j=k_0+1}^{k}\overline{\widehat{{u_j}}}(\xi_j) \right),
\end{align*}
where we used the relation $\xi=\xi_1+\dots+\xi_{k_0}-\xi_{k_0+1}\dots-\xi_k$ to write $m(\Xi)$ in terms of $\xi_1,\dots, \xi_k$ only.
We define the corresponding phase function as
\begin{equation}\label{eq:phase}
	\Phi(\xi_1,\dots,\xi_k):=\phi(\xi) - \sum_{j=1}^{k_0} \phi(\xi_j) + \sum_{j=k_0+1}^{k} \phi(\xi_j)
\end{equation}
and the convolution hyperplanes
%$$
%\Gamma_\xi = \left\{ (\xi_1,\dots,\xi_k)\in (\R^d)^k: \xi=\sum_{j=1}^{k_0} \xi_j - \sum_{j=k_0+1}^{k} \xi_j \right\},\  \Gamma_\tau = \left\{ (\tau_1,\dots,\tau_k)\in \R^k: \tau=\sum_{j=1}^{k_0} \tau_j - \sum_{j=k_0+1}^{k} \tau_j \right\}
%$$
$$
\Gamma_\xi = \left\{ (\xi_1,\dots,\xi_k)\in (\R^d)^k: \xi=\sum_{j=1}^{k_0} \xi_j - \sum_{j=k_0+1}^{k} \xi_j \right\},$$
$$\Gamma_\tau = \left\{ (\tau_1,\dots,\tau_k)\in \R^k: \tau=\sum_{j=1}^{k_0} \tau_j - \sum_{j=k_0+1}^{k} \tau_j \right\}.
$$
\begin{nb}
	In this work, we will be mainly interested in the Schrödinger Bourgain space $X_S^{s,b}$, where
	$$
	\phi(\xi)=\sum_{i=1}^d x_i^2,\quad \xi=(x_1,\dots,x_d),
	$$
	the Zakharov-Kuznetsov Bourgain space $X_{ZK}^{s,b}$, with
	\begin{equation}\label{phaseZK}
	\phi(\xi)=x_1\sum_{i=1}^d x_i^2,\quad \xi=(x_1,\dots,x_d), \ d\ge 2,
	\end{equation}
	and the Korteweg-de Vries Bourgain space $X_{KdV}^{s,b}$, where
	$$
	\phi(\xi)=\xi^3.
	$$
\end{nb}
Given $s,s'\in\R$ and $-1/2<b-1<b'<0$, the next lemmas give sufficient conditions for the validity of the estimate
\begin{equation}\label{eq:multilinear}
\|N[u_1,\dots,u_k]\|_{X^{s',b'}} \lesssim \prod_{j=1}^k\|u_j\|_{X^{s,b}}.
\end{equation}

\begin{lem}\label{lem:CS}
	Suppose that there exists $\beta<-b'$ such that, for any $M>1$,
	\begin{equation}
		\sup_{\xi} \int_{\Gamma_\xi}\frac{|m( \xi_1,\dots,\xi_k )|^2\jap{\xi}^{2s'}}{\prod_{j=1}^k\jap{\xi_j}^{2s}}\fiM d\xi_1\dots d \xi_{k-1}\lesssim M^{2\beta}.
	\end{equation}
	Then \eqref{eq:multilinear} holds. 
\end{lem}
\begin{nb}
	The result is also valid if one switches $\xi$ with one of the $\xi_j$.
\end{nb}
\begin{proof}
	Arguing by duality as in \eqref{eq:duality}, but now considering the conjugates of $u$ in the nonlinearity, it suffices to show that
	\begin{equation*}
	I=\left|\int_{\Gamma_\xi,\Gamma_\tau}\frac{m( \xi_1,\dots,\xi_k )\jap{\xi}^{s'}\jap{\tau-\phi(\xi)}^{b'}}{\prod_{j=1}^k\jap{\xi_j}^s\jap{\tau_j-\phi_j(\xi_j)}^b}v\prod_{j=1}^{k}v_jd\xi_1\dots d \xi_{k}d\tau_1\dots d\tau_k\right|
	\lesssim \|v\|_{L^2}\prod_{j=1}^k\|v_j\|_{L^2}.
	\end{equation*}
By Cauchy-Schwarz,
\begin{align*}
	I&\lesssim \left(\int_{\Gamma_\xi,\Gamma_\tau}\frac{|m( \xi_1,\dots,\xi_k )|^2\jap{\xi}^{2s'}\jap{\tau-\phi(\xi)}^{2b'}}{\prod_{j=1}^k\jap{\xi_j}^{2s}\jap{\tau_j-\phi_j(\xi_j)}^{2b}}|v|^2d\xi_1\dots d \xi_{k}d\tau_1\dots d\tau_k\right)^{\frac{1}{2}}\prod_{j=1}^k\|v_j\|_{L^2}\\&\lesssim\sup_{\tau,\xi} \left(\int_{\Gamma_\xi,\Gamma_\tau}\frac{|m( \xi_1,\dots,\xi_k )|^2\jap{\xi}^{2s'}\jap{\tau-\phi(\xi)}^{2b'}}{\prod_{j=1}^k\jap{\xi_j}^{2s}\jap{\tau_j-\phi_j(\xi_j)}^{2b}}d\xi_1\dots d \xi_{k-1}d\tau_1\dots d\tau_{k-1}\right)^{\frac{1}{2}}
	\|v\|_{L^2}\prod_{j=1}^k\|v_j\|_{L^2}.
\end{align*}
Since $2b>1$, integrating in $\tau_1$ to $\tau_{k-1}$ and using the fact that
$$\int \frac{1}{\jap{\tau-a_1}^{2b}\jap{\tau-a_2}^{2b}}d\tau\lesssim \frac{1}{\jap{a_1-a_2}^{2b}},$$
for any $a_1,a_2 \in \R$, in every one of those integrals, yields 
\begin{equation*}
\int_{\Gamma_\tau} \frac{1}{\prod_{j=1}^{k}\jap{\tau_j-\phi_j(\xi_j)}^{2b}}d\tau_1\dots  d\tau_{k-1}  \lesssim \frac{1}{\jap{\tau-\sum_{j=1}^{k_0}\phi(\xi_j) + \sum_{k_0+1}^k \phi(\xi_j)}^{2b}}= \frac{1}{\jap{\tau-\phi(\xi) + \Phi}^{2b}}.
\end{equation*}
and thus, using the simple estimate $\jap{a_1+a_2}\lesssim \jap{a_1}\jap{a_2}$, for any $a_1,a_2 \in \R$, we get
$$
\int_{\Gamma_\tau} \frac{\jap{\tau-\phi(\xi)}^{2b'}}{\prod_{j=1}^{k}\jap{\tau_j-\phi_j(\xi_j)}^{2b}} d\tau_1\dots  d\tau_{k-1} \lesssim \frac{\jap{\tau-\phi(\xi)}^{2b'}}{\jap{\tau-\phi(\xi) + \Phi}^{2b}} \lesssim \frac{1}{\jap{\Phi}^{-2b'}}.
$$
Performing a dyadic decomposition in $\Phi$,
\begin{align*}
	I&\lesssim \left(\int_{\Gamma_\xi}\frac{|m( \xi_1,\dots,\xi_k )|^2\jap{\xi}^{2s'}}{\jap{\Phi}^{-2b'}\prod_{j=1}^k\jap{\xi_j}^{2s}}d\xi_1\dots d \xi_{k}\right)^{\frac{1}{2}}\|v\|_{L^2}\prod_{j=1}^k\|v_j\|_{L^2} \\&\lesssim \sum_{M\ dyadic}\left(\int_{\Gamma_\xi}\frac{|m( \xi_1,\dots,\xi_k )|^2\jap{\xi}^{2s'}\fiM}{M^{-2b'}\prod_{j=1}^k\jap{\xi_j}^{2s}}d\xi_1\dots d \xi_{k}\right)^{\frac{1}{2}}\|v\|_{L^2}\prod_{j=1}^k\|v_j\|_{L^2} \\&\lesssim \sum_{M\ dyadic} \frac{M^{\beta}}{M^{-b'}}\|v\|_{L^2}\prod_{j=1}^k\|v_j\|_{L^2} \lesssim \|v\|_{L^2}\prod_{j=1}^k\|v_j\|_{L^2}.
\end{align*}
\end{proof}

In the next lemma, we write $\xi=\xi_0$ and $\tau=\tau_0$. Given a set of indices $A\subset \{0,\dots, k\}$, we abbreviate ``$\xi_j, \ j\in A$'' as ``$\xi_{j\in A}$''.
\begin{lem}\label{lem:interpol}
  Suppose that there exist $\emptyset\neq A \subsetneq \{0,\dots,k\}$, $\mathcal{M}_j=\mathcal{M}_j(\xi_1,\dots,\xi_k)\ge 0$, $j=1,2$, and  $\beta<-2b'$ such that
  $$
  \left(\mathcal{M}_1\mathcal{M}_2\right)^{\frac{1}{2}}=\frac{|m( \xi_1,\dots,\xi_k )|\jap{\xi}^{s'}}{\prod_{j=1}^k\jap{\xi_j}^{s}}
  $$
  and, for any $M>1$,
	\begin{equation}
		\sup_{\xi_{j\in A},\alpha} \int_{\Gamma_\xi}\mathcal{M}_1\fia d\xi_{j\notin A} + 	\sup_{\xi_{j\notin A},\alpha} \int_{\Gamma_\xi}\mathcal{M}_2\fia d\xi_{j\in A}  \lesssim M^{\beta}.
	\end{equation}
	Then \eqref{eq:multilinear} holds. 
\end{lem}

\begin{proof}
	For the sake of clarity, suppose that $k_0=k$ and that $0\in A$. Consider the operator $T_A:(\mathcal{S}(\R\times \R^d))^{k+1}\to \R$ defined as
	$$
	T_A[v_0,\dots, v_k]= \int_{\Gamma_\xi, \Gamma_\tau} \mathcal{M}_1\frac{\jap{\tau_0-\phi(\xi_0)}^{2b'}}{\prod_{j\in A, j\ge 1}\jap{\tau_j-\phi(\xi_j)}^{2b}} v_0(\tau_0,\xi_0)\dots v_k(\tau_k,\xi_k) d\xi_1\dots d\xi_k d\tau_1\dots d\tau_k.
	$$
	Suppose that $v_{j\in A}\in L^1(\R\times \R^d)$ and $v_{j\notin A}\in L^\infty(\R\times \R^d)$, with unit norms. Then
\begin{align*}
		|T_A[v_0,\dots, v_k]|&\lesssim \int_{\Gamma_\xi, \Gamma_\tau} \mathcal{M}_1\frac{\jap{\tau_0-\phi(\xi_0)}^{2b'}}{\prod_{j\in A, j\ge 1}\jap{\tau_j-\phi(\xi_j)}^{2b}} \prod_{j\in A} |v_j(\tau_j,\xi_j)| d\xi_1\dots d\xi_k d\tau_1\dots d\tau_k\\&\lesssim \sup_{\tau_{j\in A}, \xi_{j\in A}} \int_{\Gamma_\xi, \Gamma_\tau} \mathcal{M}_1\frac{\jap{\tau_0-\phi(\xi_0)}^{2b'}}{\prod_{j\in A, j\ge 1}\jap{\tau_j-\phi(\xi_j)}^{2b}} d\xi_{j\notin A}d\tau_{j\notin A}.
\end{align*}
Integrating first in the $\tau$ variables,
$$
\int_{\Gamma_\tau} \frac{\jap{\tau_0-\phi(\xi_0)}^{2b'}}{\prod_{j\in A, j\ge 1}\jap{\tau_j-\phi(\xi_j)}^{2b}} d\tau_{j\notin A} \lesssim \jap{\Phi-\sum_{j\in A} (\tau_j - \phi(\xi_j))}^{2b'}
$$
and thus
\begin{align*}
	|T_A[v_0,\dots, v_k]|&\lesssim \sup_{\tau_{j\in A}, \xi_{j\in A}} \int_{\Gamma_\xi} \mathcal{M}_1\jap{\Phi-\sum_{j\in A} (\tau_j - \phi(\xi_j))}^{2b'} d\xi_{j\notin A}\\&\lesssim \sup_{ \xi_{j\in A}, \alpha\in\R} \int_{\Gamma_\xi} \mathcal{M}_1\jap{\Phi-\alpha}^{2b'} d\xi_{j\notin A}.
\end{align*}
Performing a dyadic decomposition in $\Phi-\alpha$,
\begin{align*}
		|T_A[v_0,\dots, v_k]|\lesssim \sup_{ \xi_{j\in A}, \alpha\in\R}\sum_{M\ dyadic} M^{2b'} \int_{\Gamma_\xi} \mathcal{M}_1\fia  d\xi_{j\notin A}\lesssim \sum_{M\ dyadic} M^{2b'+\beta} \lesssim 1.
\end{align*}
Therefore $T_A$ is a bounded operator for $v_{j\in A}\in L^1(\R\times \R^d)$ and $v_{j\notin A}\in L^\infty(\R\times \R^d)$. Analogously, $T_{A^c}$ defined as
	$$
T_{A^c}[v_0,\dots, v_k]= \int_{\Gamma_\xi, \Gamma_\tau} \mathcal{M}_2\frac{1}{\prod_{j\notin A}\jap{\tau_j-\phi(\xi_j)}^{2b}} v_0(\tau_0,\xi_0)\dots v_k(\tau_k,\xi_k) d\xi_1\dots d\xi_k d\tau_1\dots d\tau_k
$$
is a bounded operator for $v_{j\notin A}\in L^1(\R\times \R^d)$ and $v_{j\in A}\in L^\infty(\R\times \R^d)$. By interpolation, it follows that
$$\left|\int_{\Gamma_\xi,\Gamma_\tau}\frac{m(\xi_1,\dots,\xi_k)\jap{\xi}^{s'}\jap{\tau-\phi(\xi)}^{b'}}{\prod_{j=1}^k\jap{\xi_j}^s\jap{\tau_j-\phi(\xi_j)}^b}v\prod_{j=1}^{k}v_jd\xi_1\dots d \xi_{k}d\tau_1\dots d\tau_k\right|\lesssim \|v\|_{L^2}\prod_{j=1}^k\|v_j\|_{L^2},$$
which is the duality version of \eqref{eq:multilinear}.
\end{proof}
\begin{nb}\label{nota:M}
	Performing a careful inspection of the above proof, one can see that $\mathcal{M}_1, \mathcal{M}_2$ can also depend on the dyadic block $M$ (since $|\Phi-\alpha| \sim M$).
\end{nb}

The arguments in Lemmas \ref{lem:CS} and \ref{lem:interpol} can be combined in several ways to give other sufficient conditions for \eqref{eq:multilinear}. In this work, we will need the following anisotropic estimate:
\begin{lem}\label{lem:interpol2}
Fix $d=2$ and write $\xi=(x,y)$. Suppose that there exist $\emptyset\neq A \subsetneq \{0,\dots,k\}$ and  $\beta<-b'$ such that, for any $M>1$,
\begin{equation}
	\sup_{x,y_{j\in A},\alpha} \int_{\Gamma_y}\left(\int_{\Gamma_x}\frac{|m(\xi_1,\dots,\xi_k)|^2\jap{\xi}^{2s'}}{\prod_{j=1}^k\jap{\xi_j}^{2s}}\fiM dx_1\dots dx_{k-1} \right)^\frac{1}{2}dy_{j\notin A} \lesssim M^{\beta}
\end{equation}
and
\begin{equation}
	\sup_{x,y_{j\notin A},\alpha} \int_{\Gamma_y}\left(\int_{\Gamma_x}\frac{|m(\xi_1,\dots,\xi_k)|^2\jap{\xi}^{2s'}}{\prod_{j=1}^k\jap{\xi_j}^{2s}}\fiM dx_1\dots dx_{k-1} \right)^\frac{1}{2}dy_{j\in A}  \lesssim M^{\beta}.
\end{equation}
Then \eqref{eq:multilinear} holds.
\end{lem}
\begin{proof}
	Suppose that $v_{j\in A}\in L^1_yL^2_{\tau,x}$ and $v_{j\notin A}\in L^\infty_yL^2_{\tau,x}$, with unit norms. Then, using Cauchy-Schwarz in the $\tau,x$ variables,
	\begin{align*}
		&\left|\int_{\Gamma_\xi,\Gamma_\tau}\frac{m\jap{\xi}^{s'}\jap{\tau-\phi(\xi)}^{b'}}{\prod_{j=1}^k\jap{\xi_j}^s\jap{\tau_j-\phi(\xi_j)}^b}v\prod_{j=1}^{k}v_jd\xi_1\dots d \xi_{k}d\tau_1\dots d\tau_k\right|\\ \lesssim & \int_{\Gamma_y}\sup_{\tau,x} \left(\int_{\Gamma_x,\Gamma_\tau}\frac{|m|^2\jap{\xi}^{2s'}\jap{\tau-\phi(\xi)}^{2b'}}{\prod_{j=1}^k\jap{\xi_j}^{2s}\jap{\tau_j-\phi(\xi_j)}^{2b}}dx_1\dots d x_{k-1}d\tau_1\dots d\tau_{k-1}\right)^{\frac{1}{2}}\prod_{j=0}^k\|v_j(y_j)\|_{L^2_{\tau,x}}dy_1\dots dy_k\\\lesssim &\sup_{y_{j\in A}}\int_{\Gamma_y}\sup_{\tau,x} \left(\int_{\Gamma_x,\Gamma_\tau}\frac{|m|^2\jap{\xi}^{2s'}\jap{\tau-\phi(\xi)}^{2b'}}{\prod_{j=1}^k\jap{\xi_j}^{2s}\jap{\tau_j-\phi(\xi_j)}^{2b}}dx_1\dots d x_{k-1}d\tau_1\dots d\tau_{k-1}\right)^{\frac{1}{2}}dy_{j\notin A}
	\end{align*}
Following the proof of Lemma \ref{lem:CS}, this quantity is bounded. Exchanging $A$ with $A^c$, the result now follows by interpolation.
\end{proof}

Finally, it is often sufficient to prove frequency-restricted estimates for lower dimensions and then extend them to the higher-dimensional case. To state precisely the result, let us consider $\R^d=\R^{d_1}\times \R^{d_2}$ and write an element of $\R^d$ as $\xi=(\eta,\zeta)$. Suppose that the phase function $\Phi$ may be decomposed as
$$
\Phi(\xi_1,\dots,\xi_k)=\Phi^\eta(\eta_1,\dots,\eta_k) + \Phi^\zeta(\zeta_1,\dots, \zeta_k)
$$
and that, for some $\mathcal{M}_1, \mathcal{M}_2:\R^{d_1}\to \R^+$,  $\emptyset\neq A \subsetneq \{0,\dots,k\}$ and $\beta<1$,
\begin{equation}\label{eq:reducedim}
	\sup_{\alpha,\eta_{j\in A}}\int_{\Gamma_\eta} \mathcal{M}_1\mathbbm{1}_{|\Phi^\eta-\alpha|<M}d\eta_{j\notin A} + 	\sup_{\alpha,\eta_{j\notin A}}\int_{\Gamma_\eta} \mathcal{M}_2\mathbbm{1}_{|\Phi^\eta-\alpha|<M}d\eta_{j\in A}\lesssim M^{\beta},\quad M>1.
\end{equation}
\begin{lem}\label{lem:descent}
	Let $\mathcal{K}_1,\mathcal{K}_2:\R^{d_1+d_2}\to \R^+$ be such that
	$$
	  \left(\mathcal{K}_1\mathcal{K}_2\right)^{\frac{1}{2}}=\frac{|m(\xi_1,\dots,\xi_k)|\jap{\xi}^{s'}}{\prod_{j=1}^k\jap{\xi_j}^{s}},
	$$
	$$
	\mathcal{K}_1(\xi_1,\dots,\xi_k)\lesssim \mathcal{M}_1(\eta_1,\dots, \eta_k)\max_{l\notin A}\left\{\prod_{j\notin A, j\neq l}\frac{1}{\jap{\zeta_j}^{d_2^+}}\right\}, 	
	$$
	and
	$$
	\mathcal{K}_2(\xi_1,\dots,\xi_k)\lesssim \mathcal{M}_2(\eta_1,\dots, \eta_k)\max_{l\in A}\left\{\prod_{j\in A, j\neq l}\frac{1}{\jap{\zeta_j}^{d_2^+}}\right\}.
	$$
	Then, if \eqref{eq:reducedim} holds for $\beta<-2b'$, then \eqref{eq:multilinear} also holds.
\end{lem}

\begin{proof}
	It suffices to apply Lemma \ref{lem:interpol} to $\mathcal{K}_1$ and $\mathcal{K}_2$. Indeed, for $\alpha\in\R$ fixed,
	\begin{align*}
		\sup_{\xi_{j\in A}} \int_{\Gamma_\xi} \mathcal{K}_1\mathbbm{1}_{|\Phi^\eta + \Phi^\zeta-\alpha|<M}d\xi_{j\notin A} &\lesssim \sup_{\xi_{j\in A}} \int_{\Gamma_\xi} \mathcal{M}_1\max_{l\notin A}\left\{\prod_{j\notin A, j\neq l}\frac{1}{\jap{\zeta_j}^{d_2^+}}\right\}\mathbbm{1}_{|\Phi^\eta + \Phi^\zeta-\alpha|<M}d\eta_{j\notin A}d\zeta_{j\notin A} \\&\lesssim \sup_{\xi_{j\in A}} \int_{\Gamma_\zeta}\left(\int _{\Gamma_\eta} \mathcal{M}_1\mathbbm{1}_{|\Phi^\eta + \Phi^\zeta-\alpha|<M}d\eta_{j\notin A}\right)\max_{l\notin A}\left\{\prod_{j\notin A, j\neq l}\frac{1}{\jap{\zeta_j}^{d_2^+}}\right\}d\zeta_{j\notin A}\\&\lesssim \sup_{\eta_{j\in A},\tilde{\alpha}} \int_{\Gamma_\eta} \mathcal{M}_1\mathbbm{1}_{|\Phi^\eta - \tilde{\alpha}|<M}d\eta_{j\notin A} \lesssim M^\beta.
	\end{align*}
	and the estimate for $\mathcal{K}_2$ is completely analogous.
\end{proof}

\begin{nb}
	In practice, since one usually considers $b=\frac{1}{2}^+$ and $b'=(b-1)^+$, it suffices to take $\beta=\frac{1}{2}^-$ in Lemmas \ref{lem:CS} and \ref{lem:interpol2} and $\beta=1^-$ in Lemmas \ref{lem:interpol} and \ref{lem:descent}.
\end{nb}

\section{Multilinear estimates}

\begin{lem}\label{lem:quadraticas}
	For $N, M>0$ and  $\alpha\in \R$ fixed, 
\begin{equation}\label{eq:polares}
	\int_{|p|,|q|<N} \mathbbm{1}_{|p^2\pm q^2-\alpha|<M}dpdq\lesssim M^{1^-}N^{0^+}.
\end{equation}
and
\begin{equation}\label{eq:quadrat}
	\int \mathbbm{1}_{|p^2-\alpha|<M}dp\lesssim M^{1/2}.
\end{equation}
\end{lem}
\begin{proof}
	We begin with \eqref{eq:polares}. For the + sign, it suffices to use Hölder:
	\begin{align*}
\int_{|p|,|q|<N} \mathbbm{1}_{|p^2+q^2-\alpha|<M}dpdq\lesssim \left(\int_{|p|,|q|<N}dp_1dq_1\right)^{0^+}\left(\int \mathbbm{1}_{|p^2+q^2-\alpha|<M}dpdq\right)^{1^-}\lesssim M^{1^-}N^{0^+}.
	\end{align*}
For the $-$ sign, 
\begin{align*}
	\int_{|p|,|q|<N} \mathbbm{1}_{|p^2-q^2-\alpha|<M}dpdq&\lesssim 	\int_{|p_1|,|q_1|<2N} \mathbbm{1}_{|p_1q_1-\alpha|<M}dp_1dq_1\\&\lesssim 	\int_{|q_1|<2N}\left(\int_{|p_1|<2N}dp_1\right)^{0^+} \left(\int\mathbbm{1}_{|p_1q_1-\alpha|<M}dp_1\right)^{1^-}dq_1 \\&\lesssim N^{0^+}\int_{|q_1|<2N} \frac{M^{1^-}}{|q_1|^{1^-}}dq_1\lesssim N^{0^+}M^{1^-}.
\end{align*}
For \eqref{eq:quadrat}, the estimate is trivial if $|\alpha|\lesssim M$. If $|\alpha|\gg M$, then
$$
\int \mathbbm{1}_{|p^2-\alpha|<M}dp\lesssim (|\alpha|+M)^{1/2} - (|\alpha|-M)^{1/2} \lesssim \frac{M}{\left(|\alpha|-M\right)^{1/2}}\lesssim M^{1/2}.
$$
\end{proof}
\begin{nb}
	Throughout this section, $\alpha\in\R$ will always be fixed. It is clear from the following proofs that the implicit constants of the estimates do not depend on $\alpha.$
\end{nb}
\subsection{Estimates for the modified Zakharov-Kuznetsov equation}
In this section, we prove trilinear estimates in the Zakharov-Kuznetsov Bourgain spaces in $d$ dimensions. Define
\begin{equation}\label{eq:phix}
\Phi^x=x^3-\sum_{j=1}^3 x_j^3.
\end{equation}
We start by recalling a frequency-restricted estimate for the modified KdV. We include the proof for the sake of completeness.
\begin{lem}[\cite{CS}]\label{lem:mkdv}
	Let $s>1/4$. For $x\in \R$ fixed, define $A^x=\{(x_1,x_2)\in \R^3:|x|\simeq |x_1|\simeq |x_2|\simeq |x_3|,\ x_3=x-x_1-x_2\}$. Then, for any $M>1$,
	$$
	\sup_{x}\int_{A^x} \jap{x}^{2(1-2s)}\mathbbm{1}_{|\Phi^x-\alpha|<M} dx_1dx_2\lesssim M^{1^-}.
	$$
\end{lem}
\begin{proof}
If $|x|<1$, the integral is uniformly bounded. If $|x|>1$, it suffices to consider the case $x\simeq x_1\simeq x_2 \simeq -x_3$. Writing
$$
\Phi^x=x^3P(p_1,p_2),\quad p_j=\frac{\xi_j}{\xi},\ j=1,2,
$$
we have $p_1,p_2\simeq 1$, $P(1,1)=0$, $\nabla P (1,1) = 0$ and $\det D^2P(1,1)< 0$. Therefore, by Morse's lemma, there exists a diffeomorphism $(p_1,p_2)\mapsto (q_1,q_2)$ such that
$$
\Phi^x = x^3 (q_1^2-q_2^2).
$$
Therefore, using \eqref{eq:polares},
\begin{align*}
\int_{A^x} |x|^{2(1-2s)}\mathbbm{1}_{|\Phi^x-\alpha|<M} dx_1dx_2&\lesssim \int_{p_1,p_2\simeq 1} |x|^{2(1-2s)}\mathbbm{1}_{|x^3P-\alpha|<M} |x|^2dp_1dp_2\\&\lesssim \int_{|q_1|,|q_2|\ll 1} |x|^{4-4s}\mathbbm{1}_{|x^3(q_1^2-q_2^2)-\alpha|<M} dq_1dq_2 \lesssim M^{1^-}.
\end{align*}
\end{proof}

In the two dimensional case, the linear change of coordinates, first used in \cite{grunrockherr},
\begin{equation}\label{change}
	x'=\mu x + \lambda y, \qquad y'=\mu x -\lambda y,
\end{equation}
with $\mu=4^{-1/3}$ and $\lambda=\sqrt3 \, 4^{-1/3}$ allows the third order spatial derivative of the linear
dispersion to be symmetrized, transforming the \ref{ZK} equation into the equivalent form
\begin{equation}\label{eq:ZKsym}
	\partial_t u + \partial^3_{x'} u + \partial^3_{y'} u=\mu\,  \partial_{x'} u^3 +\mu\,  \partial_{y'} u^3.
\end{equation}
Evidently, Theorem \ref{thm:2dmzk} follows from the analogous result for equation \eqref{eq:ZKsym}.

%\begin{nb}
%	This symmetrization does not exist for dimensions greater or equal than three and we will be forced to work directly with the (mZK) equation to prove Theorem \ref{thm:3dmzk}.
%\end{nb}

We write the associated Bourgain space as $X_{symZK}^{s,b}$, whose associated phase function is
$$
\phi(\xi)=x^3+y^3, \quad \xi=(x,y).
$$
Consequently, the total phase function is simply $\Phi=\Phi^x + \Phi^y$ (see \eqref{eq:phix}). We are now in position to prove the necessary multilinear estimate for \eqref{eq:ZKsym}:

\begin{prop}\label{prop:2dzk}
	Let $d=2$, $b=\frac{1}{2}^+$ and $b'=(b-1)^+$. For $s>1/4$ and $\epsilon<\min\{2s-\frac{1}{2},1\}$,
	$$
	\|\jap{\nabla}u_1u_2u_3\|_{X_{symZK}^{s+\epsilon,b'}}\lesssim \prod_{j=1}^3\|u_j\|_{X_{symZK}^{s,b}}.
	$$ 
\end{prop}

\begin{proof}
	Without loss of generality, we may assume that $|\xi|\ge |\xi_1|\ge |\xi_2| \ge |\xi_3|$, which implies that $|\xi_1|\gtrsim |\xi|$. If $|\xi|<1$, we use Lemma \ref{lem:CS}, since
	$$
	\sup_{\xi} \left(\int \frac{\jap{\xi}^{2s+2\epsilon}}{\prod_{j=1}^3\jap{\xi_j}^{2s}}\fia d\xi_1d\xi_2 \right)^{\frac{1}{2}}\lesssim \left(\int_{|\xi_1|, |\xi_2|< 1}  d\xi_1d\xi_2 \right)^{\frac{1}{2}}\lesssim 1.
	$$
	If $|\xi|>1$, we split the proof in several cases:
	
	\noindent\textbf{Case A.} $|\xi_3|\gtrsim |\xi|$, that is, all frequencies are comparable. Write
	\begin{align*}
		\Phi^x &= (x-x_2)(x-x_1)(x_1+x_2) = C_1(x,x_2) - (x-x_2)\left(x_1-\frac{x-x_2}{2}\right)^2 \\&= C_2(x_1,x_3) + (x_1+x_3)\left(x_2 + \frac{x_1+x_3}{2}\right)^2.
	\end{align*}
	\textbf{Subcase A1.} For some $j$, $|x-x_j||y-y_j|\gtrsim |\xi|$. Without loss of generality, $j=2$. Set
	$$
	x_1= \frac{x-x_2}{2} + \frac{p_1}{\sqrt{|x-x_2|}},\qquad y_1=\frac{y-y_2}{2} + \frac{q_1}{\sqrt{|y-y_2|}}.
	$$
	Then, for some $\tilde{\alpha}=\tilde{\alpha}(\xi,\xi_2)$, \eqref{eq:polares} implies
	\begin{align*}
		\sup_{\xi,\xi_2}\int \frac{\jap{\xi}^{s+\epsilon+1}}{\prod_{j=1}^3\jap{\xi_j}^s}\fia d\xi_1 &\lesssim\sup_{\xi,\xi_2} \int\frac{|\xi|^{1+\epsilon-2s}}{\sqrt{|x-x_2||y-y_2|}}\mathbbm{1}_{|p_1^2\pm q_1^2-\tilde{\alpha}(\xi,\xi_2)|<M}dp_1dq_1\\&\lesssim\sup_{\xi,\xi_2} \int_{|p_1|,|q_1|<|\xi|^{3/2}}|\xi|^{\frac{1}{2}+\epsilon-2s}\mathbbm{1}_{|p_1^2\pm q_1^2-\tilde{\alpha}(\xi,\xi_2)|<M}dp_1dq_1\lesssim M^{1^-}.
	\end{align*}
	%where wwe used \eqref{eq:polares} in the last step. If $(x-x_2)(y-y_2)<0$,
	%\begin{align*}
	%	\sup_{\xi,\xi_2}\int \frac{\jap{\xi}^{s+1}}{\Pi_{j=1}^3\jap{\xi_j}^s}\fia d\xi_1 &\lesssim\sup_{\xi,\xi_2} \int\frac{|\xi|^{1-2s}}{\sqrt{|x-x_2||y-y_2|}}\mathbbm{1}_{|p_1^2-q_1^2-\tilde{\alpha}|<M}dp_1dq_1\\&\lesssim\sup_{\xi,\xi_2} \int_{|p_1|,|q_1|<|\xi|^{3/2}}|\xi|^{-2s}\mathbbm{1}_{|p_1^2-q_1^2-\tilde{\alpha}|<M}dp_1dq_1\lesssim M^{1^-},
	%\end{align*}
	For the other side of the interpolation, take 
	$$
	x_2=\frac{x_1+x_3}{2} + \frac{p_2}{\sqrt{|x_1+x_3|}},\qquad y_2=\frac{y_1+y_3}{2}
	+ \frac{q_2}{\sqrt{|y_1+y_3|}}$$
	Since $(x_1+x_3)(y_1+y_3)=(x-x_2)(y-y_2)$, the same computations yield
	\begin{align*}
		\sup_{\xi_1,\xi_3}\int \frac{\jap{\xi}^{s+\epsilon+1}}{\prod_{j=1}^3\jap{\xi_j}^s}\fia d\xi_2  \lesssim M^{1^-}
	\end{align*}
	and thus the estimate follows from Lemma \ref{lem:interpol}.
	
	\noindent \textbf{Subcase A2.} For all $j$, $|x-x_j||y-y_j|\ll |\xi|$. Assuming that $|\xi|\sim|x|$, we cannot have $|x-x_j|\ll |\xi|$ for all $j$ (otherwise $x=\sum x_j \simeq 3x$). Without loss of generality, $|x-x_2|\gtrsim |\xi|$ and thus $|y-y_2|\ll 1$.
	\vskip5pt
	\noindent 1. If $|x-x_1|\gtrsim |\xi|$ and $|y-y_1|\ll 1$, we perform the change of variables only in the $x$-direction. Letting
		$$
		x_1=x_1^0 + \frac{p_1}{\sqrt{|x-x_2|}},
		$$
		the application of \eqref{eq:quadrat} yields
		\begin{align*}
			\sup_{\xi,\xi_2}\int \frac{\jap{\xi}^{s+\epsilon+1}}{\prod_{j=1}^3\jap{\xi_j}^s}\fia d\xi_1 &\lesssim\sup_{\xi,\xi_2} \int_{|y-y_1|\ll 1}\frac{|\xi|^{1+\epsilon-2s}}{\sqrt{|x-x_2|}}\mathbbm{1}_{|p_1^2 \pm \Phi^y-\tilde{\alpha}(\xi,\xi_2)|<M}dp_1dy_1\\&\lesssim\sup_{\xi,\xi_2} \int_{|y-y_1|\ll 1}|\xi|^{\frac{1}{2}+\epsilon-2s}\mathbbm{1}_{|p_1^2\pm\Phi^y-\tilde{\alpha}(\xi,\xi_2)|<M}dp_1dy_1\lesssim M^{1^-}.
		\end{align*}
		Replacing $\xi_1$ with $\xi_2$, one easily obtains the other side of the interpolation in order to apply Lemma \ref{lem:interpol}.
		
		\vskip5pt
		\noindent 2. If  $|x-x_1|,\ |x-x_3|\ll |\xi|$, since $|\xi|\sim |x|$, we must have $x\simeq x_1\simeq x_3 \simeq -x_2$. The estimate follows from Lemma \ref{lem:interpol2}, since, by Lemma \ref{lem:mkdv},
		\begin{align*}
			\sup_{x,y,y_1} \int_{|y-y_2|\ll1} \left(\int |x|^{2(1+\epsilon-2s)}\fia dx_1dx_2\right)^{\frac{1}{2}}dy_2 &\lesssim \sup_{x}\left(\int |x|^{2(1+\epsilon-2s)}\mathbbm{1}_{|\Phi^x-\alpha|<M} dx_1dx_2\right)^{\frac{1}{2}}\\&\lesssim M^{1/2^-}
		\end{align*}
		and analogously
		\begin{align*}
			\sup_{x,y_2,y_3} \int_{|y_1+y_3|\ll1} \left(\int |x|^{2(1+\epsilon-2s)}\fia dx_1dx_2\right)^{\frac{1}{2}}dy_1 &\lesssim \sup_{x}\left(\int |x|^{2(1+\epsilon-2s)}\mathbbm{1}_{|\Phi^x-\alpha|<M} dx_1dx_2\right)^{\frac{1}{2}}\\&\lesssim M^{1/2^-}.
		\end{align*}

	\noindent \textbf{Case B.} $|\xi_3|\ll |\xi_1|$. Assuming that $|x_1|\sim |\xi_1|$, we claim that, up to a permutation of $\xi_1$ and $\xi_2$, we must have
	\begin{equation}\label{eq:nonstat}
		|x_1^2-x_3^2|\gtrsim |\xi|^2,\ |x^2-x_2^2|\gtrsim |\xi|^2.
	\end{equation}
	The first inequality holds: if not, since $|\xi|\sim |\xi_1|\sim |x_1|$, we would have $|x_1|\sim |x_3|$, contradicting $|\xi_1|\ll |\xi_3|$. If $|x^2-x_2^2|\ll |\xi|^2$, let us see that \eqref{eq:nonstat} holds with $x_1$ and $x_2$ interchanged. Indeed, since $|x_3|\ll |x_1|\sim |\xi|$,
	$$
	|x_1+x_3||x_1+2x_2+x_3|=|x^2-x_2^2|\ll|\xi|^2\sim |x_1|^2\mbox{ implies }x_1\simeq -2x_2,\ x\simeq -x_2.
	$$
	Therefore
	$$
	|x_2^2-x_3^3|\sim |x_2|^2 \sim |\xi|^2,\quad |x^2-x_1^2|\sim 3|x|^2 \sim |\xi|^2
	$$
	and the claim follows. Using Hölder and performing the change of variables $x_1\mapsto \Phi^x$,
	\begin{align*}
		\sup_{\xi,\xi_2}\int \frac{\jap{\xi}^{s+\epsilon+1}}{\jap{\xi_1}^s\jap{\xi_3}^{2s}}\fia d\xi_1 &\lesssim \sup_{\xi,\xi_2}\left(\int_{|y_1|\lesssim |\xi|} \frac{\jap{\xi}^{1^++\epsilon}}{\jap{y_3}^{2s}}\mathbbm{1}_{|\Phi^x+\Phi^y-\alpha|<M} dx_1dy_1\right)^{1^-} \\&\lesssim \sup_{\xi,\xi_2}\left(\int_{|y_1|\lesssim |\xi|} \frac{\jap{\xi}^{1^++\epsilon}}{\jap{y_3}^{2s}}\mathbbm{1}_{|\Phi^x+\Phi^y-\alpha|<M} \frac{1}{|\xi|^2}d\Phi^xdy_1\right)^{1^-} \\&\lesssim \sup_{\xi,\xi_2}\left(\int_{|y_1|\lesssim |\xi|} \frac{1}{\jap{y-y_1-y_2}^{2s}|\xi|^{1^--\epsilon}}dy_1\right)^{1^-}M^{1^-} \lesssim M^{1^-}.
	\end{align*}
	Analogously,
	$$
	\sup_{\xi_1,\xi_3}\int \frac{\jap{\xi}^{s+\epsilon+1}}{\jap{\xi_1}^s\jap{\xi_2}^{2s}}\fia d\xi_2 \lesssim  M^{1^-}
	$$
	and the proof follows from Lemma \ref{lem:interpol}.
	%\begin{enumerate}
	%	\item If $|x-x_1|\gtrsim |\xi|$ and $|y-y_1|\ll 1$, we perform the change of varibles only in the $x$-direction. Letting
	%$$
	%x_1=x_1^0 + \frac{p_1}{\sqrt{|x-x_2|}},
	%$$
	%the application of \eqref{eq:quadrat} yields
	%\begin{align*}
	%	\sup_{\xi,\xi_2}\int \frac{\jap{\xi}^{s+1}}{\Pi_{j=1}^3\jap{\xi_j}^s}\fia d\xi_1 &\lesssim\sup_{\xi,\xi_2} \int_{|y-y_1|\ll 1}\frac{|\xi|^{1-2s}}{\sqrt{|x-x_2|}}\mathbbm{1}_{|p_1^2 \pm \Phi^y-\tilde{\alpha}|<M}dp_1dy_1\\&\lesssim\sup_{\xi,\xi_2} \int_{|y-y_1|\ll 1}|\xi|^{\frac{1}{2}-2s}\mathbbm{1}_{|p_1^2\pm\Phi^y-\tilde{\alpha}|<M}dp_1dy_1\lesssim M^{1^-}.
	%\end{align*}
	%Replacing $\xi_1$ with $\xi_2$, one easily obtains the other side of the interpolation.
	%
	%\item If  $|x-x_1|,\ |x-x_3|\ll |\xi|$, since $|\xi|\sim |x|$, we must have $x\sim x_1\sim x_2 \sim -x_2$. Applying the second interpolation argument and Lemma \ref{lem:mkdv},
	%\begin{align*}
	%	\sup_{x,y,y_1} \int_{|y-y_2|\ll1} \left(\int |x|^{2(1-2s)}\fia dx_1dx_2\right)^{\frac{1}{2}}dy_2 &\lesssim \sup_{x,\alpha}\left(\int |x|^{2(1-2s)}\mathbbm{1}_{|\Phi^x-\alpha|<M} dx_1dx_2\right)^{\frac{1}{2}}\\&\lesssim M^{1/2}
	%\end{align*}
	%and analogously
	%\begin{align*}
	%	\sup_{x,y_2,y_3} \int_{|y_1+y_3|\ll1} \left(\int |x|^{2(1-2s)}\fia dx_1dx_2\right)^{\frac{1}{2}}dy_1 &\lesssim \sup_{x,\alpha}\left(\int |x|^{2(1-2s)}\mathbbm{1}_{|\Phi^x-\alpha|<M} dx_1dx_2\right)^{\frac{1}{2}}\\&\lesssim M^{1/2}.
	%\end{align*}
	%\end{enumerate}
	
\end{proof}

In dimensions three and higher, no symmetrization of the linear differential operator analogous to the two dimensional one is possible, so the following proof deals with the original  \ref{ZK} equation directly, in particular considering the (non-symmetric) phase function \eqref{phaseZK}.

\begin{prop}\label{prop:3dzk}
	Fix $d\ge3$, $b=\frac{1}{2}^+$ and $b'=(b-1)^+$. For $s>d/2-1$ and $\epsilon<\min\{2s-d+2,1\}$,
	$$
	\|\jap{\nabla}u_1u_2u_3\|_{X_{ZK}^{s+\epsilon,b'}}\lesssim \prod_{j=1}^3\|u_j\|_{X_{ZK}^{s,b}}.
	$$ 
\end{prop}
\begin{proof}
Without loss of generality, we can assume that $|\xi|\gtrsim |\xi_1|\gtrsim |\xi_2|\gtrsim |\xi_3|$ (which also implies that $|\xi|\sim |\xi_1|$). If $|\xi|<1$, then all frequencies are bounded and, as in the previous proof, we apply Lemma 1. Thus we focus on the case $|\xi|>1$.

\textit{Step 1. Semi-nondegeneracy.} Given an element $v\in \R^d$, we write $v=(v^1,\dots, v^d)$. Set
$$
\Phi=|\xi|^3P\left(p_1,p_2,p_3\right),\quad p_j=\frac{\xi_j}{|\xi|},\ p=\frac{\xi}{|\xi|}.
$$
For $\xi,\xi_2$ fixed and $\xi_3=\xi-\xi_2-\xi_1$, $P$ becomes a polynomial on the components of $p_1$. A simple computation shows that
$$
\frac{\partial P}{\partial p_1^1} = 3(p_3^1)^2 + \sum_{j=2}^d (p_3^j)^2 - 3(p_1^1)^2 - \sum_{j=2}^d (p_1^j)^2,\qquad \frac{\partial P}{\partial p_1^j} = p_3^1p_3^j - p_1^1p_1^j,\ j=2,\dots, d
$$
and the Hessian has a block structure
$$
D^2P=-\begin{bmatrix}
	6(p_1^1+p_3^1) & A \\ A^T & D
\end{bmatrix}, \mbox{ where }A= \begin{bmatrix}
2(p_1^2+p_3^2)&\dots & 2(p_1^d+p_3^d)
\end{bmatrix},\ D=(p_1^1+p_3^1)I_{(d-1)\times (d-1)}.
$$
If $p_1^1\neq -p_3^1$, the rank of $D^2P$ is at least two (because of the $D$ block). On the other hand, if $p_1^1=-p_3^1$, the rank of $D^2P$ is smaller than two iff $A=0$. We conclude that $\mbox{rank}(D^2P)\le 1$ iff $p_1=-p_3$.

Similarly, if one writes
$$
\Phi=|\xi_1|^3Q(q_1,q_2,q_3),\quad q_j=\frac{\xi_j}{|\xi_1|},\ q=\frac{\xi}{|\xi_1|},
$$
fixes $\xi_1,\xi_3$ and takes $\xi_2 = \xi-\xi_1-\xi_3$,
$$
\frac{\partial Q}{\partial q^1} = 3(q^1)^2 + \sum_{j=2}^d (q^j)^2 - 3(q_2^1)^2 - \sum_{j=2}^d (q_2^j)^2,\qquad \frac{\partial Q}{\partial q^j} = q^1q^j - q_2^1q_2^j,\ j=2,\dots, d
$$
and $\mbox{rank}(D^2Q)\le 1$ iff $q=q_2$.

We claim that, up to a rearrangement of $\xi_1,\xi_2,\xi_3$, one can ensure that both $D^2P$ and $D^2Q$ have a rank of at least two. If $p_1\neq-p_3$, then
$$
q=\frac{|\xi|}{|\xi_1|}p=\frac{|\xi|}{|\xi_1|}(p_1+p_2+p_3) \neq \frac{|\xi|}{|\xi_1|}p_2=q_2
$$
and we are done. If $p_1=-p_3$ and  $q\neq q_1$, then $p_2\neq -p_3$. Therefore, exchanging $\xi_1$ and $\xi_2$, the claim also follows. Finally, if $p=p_1=-p_3$, then $q\neq q_3$ and $p_2\neq -p_1$ and it suffices to exchange $\xi_2$ with $\xi_3$.

\textit{Step 2. Application of Lemma 1.} By the previous step, we can assume that both $\mbox{rank}(D^2P)$ and $\mbox{rank}(D^2Q)$ are greater or equal than 2. We now prove the claimed estimate through Lemma 1 with $A=\{0,2\}$, interpolating between
$$
\sup_{\xi,\xi_2} \int \frac{\jap{\xi}^{s+1+\epsilon}}{\jap{\xi_1}^s\jap{\xi_3}^{2s}}\fia d\xi_1 \lesssim M^{1^-}
$$
and
$$
\sup_{\xi_1,\xi_3} \int \frac{\jap{\xi}^{s+1+\epsilon}}{\jap{\xi_1}^s\jap{\xi_2}^{2s}}\fia d\xi \lesssim M^{1^-}.
$$
We focus on the first estimate, as the second follows from analogous computations.

\noindent \textbf{Case A}. $|\xi_3|\gtrsim |\xi|$. Then 
\begin{align*}
\sup_{\xi,\xi_2} \int \frac{\jap{\xi}^{s+1+\epsilon}}{\jap{\xi_1}^s\jap{\xi_3}^{2s}}\fia d\xi_1 &\lesssim \sup_{\xi,\xi_2} \int \jap{\xi}^{1+\epsilon-2s}\fia d\xi_1 \\&\lesssim  \sup_{\xi,\xi_2} \int_{|p_1|\lesssim 1} |\xi|^{d+1+\epsilon-2s}\mathbbm{1}_{||\xi|^3P-\alpha|<M} dp_1.
\end{align*}
\textbf{Subcase A1.} $|\nabla P|\gtrsim 1$. Without loss of generality, $|\partial_{p_1^1}P|\gtrsim 1$. Then, applying Hölder and performing the change of variables $p_1^1\mapsto P$,
\begin{align*}
	\sup_{\xi,\xi_2} \int_{|p_1|\lesssim 1} |\xi|^{d+1+\epsilon-2s}\mathbbm{1}_{||\xi|^3P-\alpha|<M} dp_1 &\lesssim \sup_{\xi,\xi_2} \left(\int_{|p_1|\lesssim 1} |\xi|^{d+1^++\epsilon-2s}\mathbbm{1}_{||\xi|^3P-\alpha|<M} dPdp_1^2\dots dp_1^d\right)^{1^-}\\& \lesssim\sup_{\xi,\xi_2} \left(|\xi|^{d+1^++\epsilon-2s}\frac{M}{|\xi|^3}\right)^{1^-}\lesssim M^{1^-}.
\end{align*}
\textbf{Subcase A2.} $|\nabla P|\ll 1$. Since $P$ has a compact set of critical points, it suffices to localize around one of them, that is, fix $z\in \R^d$ such that $\nabla P(z)=0$ and $|p_1-z|<\delta\ll 1$. Since $\mbox{rank}(D^2P(z))\gtrsim 2$, the Hessian has at least two nondegenerate directions. Hence the application of Morse's Splitting Lemma (\cite[Theorem 8.3]{MawhinWillem}) implies the existence of an invertible change of variables $\phi$, from a neighborhood\footnote{As it can be seen from the proof of \cite[Theorem 8.3]{MawhinWillem}, the size of the neighborhood on which the lemma holds can be chosen to depend continuously on $p_2$. Since $p_2$ varies on a compact set, this allows us to choose a universal $\delta>0$ such that Morse's Splitting Lemma holds for every $|p_2|\le 1$ and every critical point $z$.} of $z$ to a neighborhood of $0$, such that $\tilde{p}_1=\phi(p_1)$ satisfies
$$
P(p_1) = P(z) \pm (\tilde{p}_1^1)^2 \pm (\tilde{p}_1^2)^2 + h(\tilde{p}_1^3,\dots, \tilde{p}_1^d), \mbox{ for some }h:\R^{d-2}\to \R.
$$
Then
\begin{align*}
	&\sup_{\xi,\xi_2} \int_{|p_1-z|\ll1} |\xi|^{d+1+\epsilon-2s}\mathbbm{1}_{||\xi|^3P-\alpha|<M} dp_1  \\\lesssim& \sup_{\xi,\xi_2} \left(\int_{|\tilde{p}_1|\ll1} |\xi|^{d+1^++\epsilon-2s}\mathbbm{1}_{||\xi|^3(\pm (\tilde{p}_1^1)^2 \pm (\tilde{p}_1^2)^2 + h(\tilde{p}_1^3,\dots, \tilde{p}_1^d))-\alpha|<M} d\tilde{p}_1^1d\tilde{p}_1^2\dots d\tilde{p}_1^d\right)^{1^-} \\\lesssim& \sup_{\xi,\xi_2} \left(\int_{|\tilde{p}_1|\ll1} \left(\int|\xi|^{d+1^++\epsilon-2s}\mathbbm{1}_{||\xi|^3(\pm (\tilde{p}_1^1)^2 \pm (\tilde{p}_1^2)^2 + h(\tilde{p}_1^3,\dots, \tilde{p}_1^d))-\alpha|<M} d\tilde{p}_1^1d\tilde{p}_1^2\right) d\tilde{p}_1^3\dots d\tilde{p}_1^d\right)^{1^-}\\  \lesssim&\sup_{\xi,\xi_2} \left(|\xi|^{d+1^++\epsilon-2s}\frac{M}{|\xi|^3}\int_{|\tilde{p}_1|\ll1}  d\tilde{p}_1^3\dots d\tilde{p}_1^d\right)^{1^-}\lesssim M^{1^-}.
\end{align*}

\noindent \textbf{Case B.} $|\xi_3|\ll |\xi|\sim |\xi_1|$. This implies that $|\partial_{p_1^1} P|\gtrsim 1$ and we can perform the change of variables $p_1^1\mapsto P$:
\begin{align*}
	\sup_{\xi,\xi_2} \int \frac{\jap{\xi}^{s+1+\epsilon}}{\jap{\xi_1}^s\jap{\xi_3}^{2s}}\fia d\xi_1 & \lesssim \sup_{\xi,\xi_2} \int_{|p_1|\lesssim 1} \frac{|\xi|^{1+\epsilon}}{\jap{\xi_3}^{2s}}\mathbbm{1}_{||\xi|^3P-\alpha|<M} |\xi|^ddp_1\\&\lesssim \sup_{\xi,\xi_2} \left(\int_{|p_1|\lesssim 1} \frac{|\xi|^{2^++\epsilon}}{\jap{\xi_3}^{2s}}\mathbbm{1}_{||\xi|^3P-\alpha|<M} |\xi|^{d-1}dPdp_1^2\dots dp_1^d\right)^{1^-}\\&\lesssim \sup_{\xi,\xi_2} \left(\int_{|p_1|\lesssim 1} \frac{|\xi|^{2^++\epsilon}}{\jap{(\xi_3^2,\dots,\xi_3^d)}^{2s}}\mathbbm{1}_{||\xi|^3P-\alpha|<M}dPd\xi_1^2\dots d\xi_1^d\right)^{1^-}\\&\lesssim \sup_{\xi,\xi_2} \left(\int_{|(\xi_3^2,\dots,\xi_3^d)|\ll |\xi|} \frac{1}{\jap{(\xi_3^2,\dots,\xi_3^d)}^{2s}|\xi|^{1^--\epsilon}}d\xi_1^2\dots d\xi_1^d\right)^{1^-}M^{1^-}\\&\lesssim M^{1^-}.
\end{align*}

	%\textbf{Subcase B2a.} $|y|\gtrsim Y$. Then $y\simeq y_2$. If $|y-y_1|\ll Y$, this means that
	%$$
	%y\simeq y_1 \simeq y_2 \simeq -y_3,
	%$$
	%If $|y-y_3|\ll Y$, 
	%$$
	%y\simeq y_3 \simeq y_2 \simeq -y_1,
	%$$
	%In both cases, we are close to a stationary point in the $y$-component. If $|y-y_1|, |y-y_3|\gtrsim Y$, the definition of Case B implies that
	%$$
	%x\simeq x_1\simeq x_3\simeq -x_2
	%$$ 
	%and we are close to a stationary point in the $x$-component.
\end{proof}

%\begin{nb}
%	Under the conditions of Proposition \ref{prop:3dzk}, one can check in the above proof that, up to permutations $\xi_1\leftrightarrow\xi_2$,
%	\begin{equation}\label{eq:3dzk1}
%		\sup_{\xi,\xi_2}\int \frac{\jap{\xi}^{1+\epsilon}}{\jap{\xi_3}^{2s}}\fia d\xi_1 \lesssim M^{1^-}
%	\end{equation}
%	and
%	\begin{equation}\label{eq:3dzk2}
%		\sup_{\xi_1,\xi_3}\int \frac{\jap{\xi}^{1+\epsilon}}{\jap{\xi_2}^{2s}}\fia d\xi_2 \lesssim M^{1^-}.
%	\end{equation}
%\end{nb}

\begin{nb}
In two dimensions, one cannot guarantee the semi-nondegeneracy $\mbox{rank}(D^2P)\ge 2$, because the diagonal matrix $D$ is just a $1\times 1$ matrix. This is to be expected: indeed, if the above argument worked for $d=2$, we would be able to reach the scaling regularity $s=0$, contradicting the ill-posedness result of \cite{kinomzk}.
\end{nb}

\subsection{Estimates for the nonlinear Schrödinger equation}
%Define
%$$
%\Phi^x=x^2-x_1^2+x_2^2-x_3^2
%$$

\begin{prop}\label{prop:estNLScubic}
	Fix $d=2$.  Take $b=\frac{1}{2}^+$ and $b'=(b-1)^+$. For $s>0$ and $\epsilon<\min\{2s,1\}$,
	$$
	\|u_1\overline{u_2}u_3\|_{X_{S}^{s+\epsilon,b'}}\lesssim \prod_{j=1}^3\|u_j\|_{X_S^{s,b}}.
	$$ 
\end{prop}
\begin{proof}
	Without loss of generality, we may assume that $|\xi|$ is the largest frequency and that $|\xi_1|\ge|\xi_3|$. The proof is imediate if $|\xi|<1$, we henceforth consider $|\xi|>1$. In this estimate, the phase function is $\Phi=\Phi^x + \Phi^y$, where
	$$
	\Phi^x=x^2-x_1^2+x_2^2-x_3^2, \qquad \Phi^y=y^2-y_1^2+y_2^2-y_3^2.
	$$
	
	\noindent \textbf{Case A.} $|\xi_1|\gtrsim |\xi_2|$. We apply Lemma \ref{lem:interpol} with $A=\{0,2\}$. We start with the estimate for
	$$
	\sup_{\xi,\xi_2}\int \frac{\jap{\xi}^{s+\epsilon}}{\jap{\xi_1}^s\jap{\xi_3}^{2s}}\fia dx_1dy_1.
	$$
	
	\noindent \textbf{Subcase A1.} $|\xi_3-\xi_1|\gtrsim |\xi|$. This means that
	$$
	\left| \frac{\partial \Phi^x}{\partial x_1} \right|\gtrsim |\xi|\qquad \mbox{or}\qquad  \left| \frac{\partial \Phi^y}{\partial y_1} \right| \gtrsim |\xi|.
	$$
	Assuming the first possibility, we apply Hölder and perform the change of variables $x_1\mapsto \Phi^x$:
	\begin{align*}
		\sup_{\xi,\xi_2}\int \frac{\jap{\xi}^{s+\epsilon}}{\jap{\xi_1}^s\jap{\xi_3}^{2s}}\fia dx_1dy_1 &\lesssim \sup_{\xi,\xi_2}\left(\int \frac{\jap{\xi}^{\epsilon^+}}{\jap{\xi_3}^{2s}}\fia dx_1dy_1\right)^{1^-}\\&\lesssim \sup_{\xi,\xi_2}\left(\int \frac{1}{\jap{y_3}^{2s}|\xi|^{1^--\epsilon}}\fia d\Phi^xdy_1\right)^{1^-}\\&\lesssim \sup_{\xi,\xi_2}\left(\int \frac{1}{\jap{y-y_1-y_2}^{2s+1^--\epsilon}}dy_1\right)^{1^-}M^{1^-}\lesssim M^{1^-}.
	\end{align*}
	
	\noindent \textbf{Subcase A2.} $|\xi_3-\xi_1|\ll |\xi|\sim |\xi_1|$, which implies that $\xi_1\simeq \xi_3$. Since
	$$
	\Phi^x = C(x,x_2) -2\left(x_1-\frac{x-x_2}{2}\right)^2,
	$$
	we can use polar coordinates
	$$
	x_1=\frac{x-x_2}{2} + r\cos\theta,\qquad y_1=\frac{y-y_2}{2} + r\sin\theta
	$$
	and thus, using Hölder and $|\xi_3|\gtrsim |\xi|$,
	\begin{align*}
		\sup_{\xi,\xi_2}\int \frac{\jap{\xi}^{s+\epsilon}}{\jap{\xi_1}^s\jap{\xi_3}^{2s}}\fia dx_1dy_1 &\lesssim \sup_{\xi,\xi_2}\left(\int \fia dx_1dy_1\right)^{1^-}\\&\lesssim \sup_{\xi,\xi_2}\left(\int r \mathbbm{1}_{|r^2-\tilde{\alpha}(\xi,\xi_2)|<M} dr\right)^{1^-}\lesssim M^{1^-}.
	\end{align*}
	For the other side of the interpolation, we consider 
	$$
	\sup_{\xi_1,\xi_3}\int \frac{\jap{\xi}^{s+\epsilon}}{\jap{\xi_1}^s\jap{\xi_2}^{2s}}\fia dx_2dy_2.
	$$
	\textbf{Subcase A1'.} $|\xi+\xi_2|\gtrsim |\xi|$, that is,
	$$
	\left| \frac{\partial \Phi^x}{\partial x_2} \right|\gtrsim |\xi|\qquad \mbox{or}\qquad  \left| \frac{\partial \Phi^y}{\partial y_2} \right| \gtrsim |\xi|.
	$$
	We proceed exactly as in Subcase A1:
	\begin{align*}
		\sup_{\xi_1,\xi_3}\int \frac{\jap{\xi}^{s+\epsilon}}{\jap{\xi_1}^s\jap{\xi_2}^{2s}}\fia dx_2dy_2 &\lesssim 	\sup_{\xi_1,\xi_3}\left(\int \frac{\jap{\xi}^{\epsilon^+}}{\jap{\xi_2}^{2s}}\fia dx_2dy_2\right)^{1^-}\\&\lesssim 	\sup_{\xi_1,\xi_3}\left(\int \frac{1}{\jap{y_2}^{2s}|\xi|^{1^--\epsilon}}\fia d\Phi^xdy_2\right)^{1^-}\lesssim M^{1^-}.
	\end{align*}
	\textbf{Subcase A2'.} $|\xi+\xi_2|\ll |\xi|$, which means that $\xi\simeq-\xi_2$. Writing
	$$
	\Phi^x=C(x_1,x_3)+2\left(x_2+\frac{x_1+x_3}{2}\right)^2
	$$
	the computations follow as in Subcase A2.
	
	\noindent\textbf{Case B.} $|\xi_2|\gg |\xi_1|$. We apply Lemma \ref{lem:interpol} with $A=\{0,1\}$, interpolating the estimates for
	$$
	\sup_{\xi,\xi_1} \int \frac{\jap{\xi}^{s+\epsilon}}{\jap{\xi_2}^s\jap{\xi_3}^{2s}}dx_2dy_2
	$$
	and
	$$
	\sup_{\xi_2,\xi_3} \int \frac{\jap{\xi}^{s+\epsilon}}{\jap{\xi_2}^s\jap{\xi_1}^{2s}}dx_1dy_1.
	$$
	We consider only the first integral. Since $|\xi|\sim |\xi_2|\gg |\xi_3|$,
	$$
	\left| \frac{\partial \Phi^x}{\partial x_2} \right|\gtrsim |\xi|\qquad \mbox{or}\qquad  \left| \frac{\partial \Phi^y}{\partial y_2} \right| \gtrsim |\xi|,
	$$
	and the proof follows exactly the same steps as in Subcase A1.
\end{proof}

\begin{prop}
	Fix $d=2$.  Take $b=\frac{1}{2}^+$ and $b'=(b-1)^+$. For $s>\frac{1}{2}$ and $\epsilon<\min\left\{4s-2,1\right\}$,
	$$
	\|u_1\overline{u_2}u_3\overline{u_4}u_5\|_{X_{S}^{s+\epsilon,b'}}\lesssim \prod_{j=1}^5\|u_j\|_{X_S^{s,b}}.
	$$ 
\end{prop}
\begin{proof}
	We write the phase function $\Phi$ as $\Phi=\Phi^x+\Phi^y$, where
	$$
	\Phi^x=x^2-x_1^2+x_2^2-x_3^2+x_4^2-x_5^2.
	$$
	Without loss of generality, $|\xi|>1$ is the largest frequency and the remaining frequencies are ordered in decreasing size\footnote{In this case, the $\xi_1$ and $\xi_5$ terms in $\Phi$ have the same sign. A careful inspection of the proof reveals that the computations are completely independent on the signs. Thus there is no loss of generality in our arguments.}. We apply Lemma \ref{lem:interpol} with $A=\{0,2,4\}$, interpolating between 
	$$
	I_1=\sup_{\xi,\xi_2,\xi_4} \int \frac{\jap{\xi}^{s+\epsilon}}{\jap{\xi_1}^s\jap{\xi_3}^{2s}\jap{\xi_5}^{2s}}\fia d\xi_1d\xi_3
	$$
	and
	$$
	I_2=\sup_{\xi_1,\xi_3,\xi_5} \int \frac{\jap{\xi}^{s+\epsilon}}{\jap{\xi_1}^s\jap{\xi_2}^{2s}\jap{\xi_4}^{2s}}\fia d\xi d\xi_2.
	$$
	As the analysis is basically the same for both integrals, we focus on $I_1$.
	%
	%\textbf{Case A.} All frequencies are comparable. We apply Lemma \ref{lem:interpol} with $A=\{0,2,4\}$.
	%$$
	%I=\sup_{\xi,\xi_2,\xi_4} \int \frac{\jap{\xi}^{s+\epsilon}}{\prod_{j=1}^5\jap{\xi_j}^s}\fia d\xi_1d\xi_3
	%$$
	
	\noindent \textbf{Case A.} Either
	$$
	\left| \frac{\partial \Phi^x}{\partial x_1} \right|\gtrsim |\xi|\qquad \mbox{or}\qquad  \left| \frac{\partial \Phi^y}{\partial y_1} \right| \gtrsim |\xi|.
	$$
	Assuming the first possibility,
	\begin{align*}
		I&\lesssim \sup_{\xi,\xi_2,\xi_4} \left(\int \frac{\jap{\xi}^{\epsilon^+}}{\jap{\xi_3}^{2s}\jap{\xi_5}^{2s}}\fia d\xi_1d\xi_3\right)^{1^-}\\&\lesssim \sup_{\xi,\xi_2,\xi_4} \left(\int_{|y_1|,|y_3|\lesssim |\xi|} \frac{|\xi|^{\epsilon^+-1}}{\jap{x_3}^{1/2^+}\jap{x_5}^{1/2^+}\jap{y_3}^{2s-1/2^+}\jap{y_5}^{2s-1/2^+}}\fia d\Phi^xdy_5dx_3dy_3\right)^{1^-}\\&\lesssim \sup_{\xi,\xi_2,\xi_4} \left(\int_{|y_1|,|y_3|\lesssim |\xi|} \frac{|\xi|^{\epsilon^+-1}}{\jap{y_3}^{2s-1/2^+}\jap{y_5}^{2s-1/2^+}}dy_3dy_5\right)^{1^-}\lesssim M^{1^-}.
	\end{align*}
	
	\noindent \textbf{Case B.} One has
	$$
	\left| \frac{\partial \Phi^x}{\partial x_1} \right| + \left| \frac{\partial \Phi^y}{\partial y_1} \right| \ll |\xi|,\mbox{ that is, } |\xi_1-\xi_5|\ll |\xi|\sim |\xi_1|.
	$$
	This implies that $\xi_1\simeq \xi_5$ and thus all frequencies are comparable. Writing
	$$
	\Phi^x=C(x,x_2,x_3,x_4)-2\left(x_1-\frac{x-x_2-x_3-x_4}{2}\right)^2,
	$$
	the change of variables
	$$
	x_1=\frac{x-x_2-x_3-x_4}{2} + r\cos\theta, \quad y_1=\frac{y-y_2-y_3-y_4}{2} + r\sin\theta
	$$
	leads to
	\begin{align*}
		I\lesssim \sup_{\xi,\xi_2,\xi_4}\int |\xi|^{\epsilon-4s}\mathbbm{1}_{|\Phi^x+\Phi^y-\alpha|<M}d\xi_1 d\xi_3 \lesssim \int |\xi_3|^{\epsilon-4s}\mathbbm{1}_{|r^2-\tilde{\alpha}(\x,\x_2,\xi_3,\xi_4)|<M}rdr d\xi_3 \lesssim M^{1^-}.
	\end{align*}
\end{proof}

For dimensions $d\ge 3$, it suffices to use the two-dimensional estimates together with Lemma \ref{lem:descent}:
\begin{prop}
	Fix $d\ge 3$, $b=\frac{1}{2}^+$ and $b'=(b-1)^+$.
	\begin{enumerate}
		\item For $s>d/2-1$ and $\epsilon<\min\{2s-d+2,1\}$, 
		$$
			\|u_1\overline{u_2}u_3\|_{X_{S}^{s+\epsilon,b'}}\lesssim \prod_{j=1}^3\|u_j\|_{X_S^{s,b}}.
		$$
		\item For $s>(d-1)/2$ and $\epsilon<\min\{4s+2-2d,1\}$,
			$$
		\|u_1\overline{u_2}u_3\overline{u_4}u_5\|_{X_{S}^{s+\epsilon,b'}}\lesssim \prod_{j=1}^5\|u_j\|_{X_S^{s,b}}.
		$$ 
	\end{enumerate}
\end{prop}
\begin{proof}
	The proof is direct, we exemplify with the first estimate.
	Write $\xi_j=(\eta_j,\zeta_j)$, $\eta_j\in\R^2$, $\zeta_j\in \R^{d-2}$, and $\Phi=\Phi^\eta+\Phi^\zeta$. Without loss of generality,  $\jap{\xi}\lesssim \jap{\eta}\lesssim \jap{\eta_1}$. Then
	$$
	\mathcal{K}_1:=\frac{\jap{\xi}^{s+\epsilon}}{\jap{\xi_1}^s\jap{\xi_3}^{2s}} \lesssim \left(\frac{\jap{\eta}^{s+\epsilon}}{\jap{\eta_1}^s\jap{\eta_3}^{2s-d^++2}}\right)\frac{1}{\jap{\zeta_3}^{d^+-2}}\lesssim \left(\frac{\jap{\eta}^{1+\epsilon}}{\jap{\eta_3}^{2s-d^++2}}\right)\frac{1}{\jap{\zeta_3}^{d^+-2}}=:\mathcal{M}_1\frac{1}{\jap{\zeta_3}^{d^+-2}}
	$$
	and
	$$
	\mathcal{K}_2:=\frac{\jap{\xi}^{s+\epsilon}}{\jap{\xi_1}^s\jap{\xi_3}^{2s}} \lesssim  \left(\frac{\jap{\eta}^{\epsilon}}{\jap{\eta_2}^{2s-d^++2}}\right)\frac{1}{\jap{\zeta_2}^{d^+-2}}=:\mathcal{M}_2\frac{1}{\jap{\zeta_2}^{d^+-2}}.
	$$
	By the proof of Proposition \ref{prop:estNLScubic}, \eqref{eq:reducedim} holds for $\mathcal{M}_1$ and $\mathcal{M}_2$. The result now follows immediately from Lemma \ref{lem:descent}.
\end{proof}
\subsection{Estimates for the quartic Korteweg-de Vries equation}
%
%Before we proceed to the proof of the multilinear estimate, we need the following version of Morse's lemma:
%
%\begin{lem}\label{lem:morseparcial}
%	Given $f:\R^2\to \R$ smooth, $f=f(x,y)$, suppose that $f(0,0)=0$, $\nabla f(0,0)=(1,0)$ and $\partial^2_yf(0,0)>0$. Then there exists $\delta, \tilde{\delta}>0$ and a diffeomorhism $\phi:B_\delta((0,0))\to B_{\tilde{\delta}}((0,0))$ such that
%	$$
%	f(x,y) = p + q^2,\quad\mbox{ where } (x,y)=\phi(p,q), \ (p,q)\in B_{\delta}((0,0)).
%	$$
%	and $\frac{\partial y}{\partial q}((0,0))\neq 0$.
%\end{lem}
%\begin{proof}
%	Given $k>0$, define $g(x)=x-kx^2$ and $h(x,y)=f(x,y)-g(x)$. For $k$ large enough, $(0,0)$ is a nondegenerate critical point of $h$. Hence, by Morse's lemma, there exists a local diffeomorphism $(x,y)\mapsto (z,w)$ around $(0,0)$ such that
%	$$
%	h(x,y)= z^2+w^2,\quad (x,y)=(x(z,w),y(z,w)).
%	$$
%	Therefore
%	$$
%	f(x(z,w),y(z,w))= z^2+w^2 + g(x(z,w)).
%	$$
%	Since the mapping $(x,y)\mapsto (z,w)$ is a diffeomorphism, either $\frac{\partial x}{\partial z}(0,0), \frac{\partial y}{\partial w}(0,0)\neq 0$ or $\frac{\partial x}{\partial w}(0,0), \frac{\partial y}{\partial z}(0,0)\neq 0$. Assuming the first possiblity, consider the system
%	$$
%	\begin{cases}
%	F_1(z,w,p,q)=z^2+g(x(z,w))-p=0\\ 	F_2(z,w,p,q)=w-q=0.
%	\end{cases}
%	$$
%	Since $\frac{\partial(F_1,F_2)}{\partial(z,w)}(0,0,0,0)=\frac{\partial x}{\partial z}(0,0)\neq 0$,   the Implicit Function theorem implies that $$(z,w)=(z(p,q), w(p,q))$$ and
%	$$
%	f(x(p,q),y(p,q))=p+q^2.
%	$$
%	Finally, $\frac{\partial y}{\partial q}((0,0)) = \frac{\partial y}{\partial w}((0,0))\neq 0$.
%\end{proof}

\begin{prop}\label{prop:4kdv}
	For $s>-1/6$, $b=\frac{1}{2}^+$, $b'=(b-1)^+$ and $\epsilon<\min\{3s+1/2,1\}$,
	$$
	\|\jap{\nabla}u_1u_2u_3u_4\|_{X_{KdV}^{s+\epsilon,b'}}\lesssim \prod_{j=1}^4\|u_j\|_{X_{KdV}^{s,b}}.
	$$ 
\end{prop}

\begin{proof}
	Once again, we consider the worst-case scenario $|\xi|\ge |\xi_1|\ge \dots \ge |\xi_4|$ and $|\xi|>1$. We take $-1/6<s<0$ (for $s>0$, the proof follows from similar computations).
	
	\textbf{Case A.} We do not have $|\xi|\simeq |\xi_1|\simeq |\xi_2|\simeq |\xi_3|\gg |\xi_4|$. We use Lemma \ref{lem:interpol} with $A=\{1,2,4\}$, interpolating between
	\begin{equation}\label{eq:4kdvinter1}
		I_1:=\sup_{\xi,\xi_3} \int \frac{|\xi|\jap{\xi}^{s+\epsilon+\frac{1}{2}}}{\jap{\xi_1}^{s+\frac{1}{2}}\jap{\xi_2}^{s+\frac{1}{2}}\jap{\xi_3}^s\jap{\xi_4}^s }\fia d\xi_1d\xi_2\lesssim M^{1^-}
	\end{equation}
	and
	\begin{equation}\label{eq:4kdvinter2}
		I_2:=\sup_{\xi_1,\xi_2,\xi_4} \int \frac{|\xi|\jap{\xi}^{s+\epsilon-\frac{1}{2}}}{\jap{\xi_1}^{s-\frac{1}{2}}\jap{\xi_2}^{s-\frac{1}{2}}\jap{\xi_3}^s\jap{\xi_4}^s}d\xi\lesssim M^{1^-}.
	\end{equation}
	We begin with \eqref{eq:4kdvinter1}, where $\xi,\xi_3$ are fixed and $\xi_4$ is the dependent variable.
	\vskip5pt	
	\noindent\textbf{Subcase A1.} $|\partial_{\xi_1}\Phi|\gtrsim |\xi|^2$. Then, since $|\xi|\sim |\xi_1|$ and $|\xi_2|\ge |\xi_3|\ge |\xi_4|$,
	\begin{align*}
		I_1\lesssim \sup_{\xi,\xi_3}\left(\int  \frac{|\xi|^{1^++\epsilon}}{\jap{\xi_2}^{3s+\frac{1}{2}^-}}\fia d\xi_1d\xi_2\right)^{1^-}\lesssim \sup_{\xi,\xi_3}\left(\int \frac{1}{\jap{\xi_2}^{3s+\frac{1}{2}^{-}}|\xi|^{1^--\epsilon}} \fia d\Phi d\xi_2\right)^{1^-} \lesssim M^{1^-}
	\end{align*}
	
	\vskip5pt	
	\noindent\textbf{Subcase A2.} $|\partial_{\xi_1}\Phi|\ll |\xi|^2$. This implies that
	$$
	|\xi_1^2-\xi_4^2|\ll |\xi_1|^2,\quad \mbox{that is, }\quad |\xi_1|\simeq |\xi_2| \simeq |\xi_3|\simeq |\xi_4|.
	$$
	Write
	$$
	\Phi=\xi^3P(\vec{p}),\quad P=1-\sum_{j=1}^4p_j^3,\quad \vec{p}=(p_1,p_2,p_3,p_4),\quad p_j=\frac{\xi_j}{\xi}.
	$$
	The condition of this subcase implies that $|\vec{p}-\vec{p}_0|\ll 1$, where $$\vec{p}_0=(p_{01},p_{02},p_{03},p_{04}), \quad p_{0j}=\pm 1,\ j=1,\dots, 4.$$ For $p_3$ fixed and $p_4=1-p_1-p_2-p_3$,
	$$
	\nabla P = 3(p_4^2-p_1^2, p_4^2-p_2^2), \quad D^2P=-6\begin{bmatrix}
		p_1+p_4       & p_4   \\
		p_4       &  p_2+p_4  
	\end{bmatrix}.
	$$
	At $\vec{p}_0$, the gradient vanishes and $\det(D^2P)\neq 0$. Therefore, by Morse's lemma with parameters \cite[Lemma C.6.1]{HormanderIII}, for every $p_3\simeq p_{03}$, there exists a unique critical point $z(p_3)=(z_1(p_3),z_2(p_3))$ and a diffeomorphism $(p_1,p_2)\mapsto(q_1,q_2)$ such that
	$$
	P(\vec{p})=P(z(p_3))\pm q_1^2 \pm q_2^2.
	$$
	Then, by \eqref{eq:polares},
	\begin{align*}
		I_1&\lesssim \sup_{\xi,\xi_3}\int |\xi|^{\frac{1}{2}+\epsilon-3s}\fia d\xi_1 d\xi_2 \lesssim  \sup_{\xi,\xi_3}\int |\xi|^{\frac{5}{2}+\epsilon-3s}\mathbbm{1}_{|\xi^3P-\alpha|<M}dp_1dp_2 
		\\&\lesssim  \sup_{\xi,\xi_3}\int |\xi|^{\frac{5}{2}+\epsilon-3s}\mathbbm{1}_{|\xi^3(P(z(p_3))\pm q_1^2 \pm q_2^2)-\alpha|<M}dq_1dq_2 \lesssim \sup_\xi |\xi|^{\frac{5}{2}+\epsilon-3s}\left(\frac{M}{\xi^3}\right)^{1^-}\lesssim M^{1^-}.
	\end{align*}
	\vskip10pt
	We now consider \eqref{eq:4kdvinter2}. Now $\xi_1,\xi_2, \xi_4$ are fixed and $\xi_3$ is the dependent variable.
	\vskip5pt	
	\noindent\textbf{Subcase A1'.} $|\partial_\xi \Phi|\gtrsim |\xi_1|^2$. Then we estimate directly
	\begin{align*}
		I_2\lesssim \sup_{\xi_1,\xi_2,\xi_4} \int |\xi_1|^{-3s+\epsilon+\frac{3}{2}}\fia d\xi \lesssim \sup_{\xi_1,\xi_2,\xi_4} \left(\int |\xi_1|^{-3s+\epsilon+\frac{3}{2}^+}\fia \frac{1}{|\xi_1|^2}d\Phi\right)^{1^-}\lesssim M^{1^-}.
	\end{align*}
	\vskip5pt	
	\noindent\textbf{Subcase A2'.} $|\partial_\xi\Phi|\ll |\xi_1|^2$, which implies that
	$$
	|\xi^2-\xi_3^2|\ll |\xi|^2,\quad \mbox{i.e.,}\quad |\xi|\simeq |\xi_1|\simeq |\xi_2|\simeq |\xi_3|.
	$$
	Since we are in Case A, $|\xi_4|\sim |\xi|$. Therefore
	$$
	\xi_4 =\xi-\sum_{j=1}^3 \xi_j \simeq k\xi,\quad k\in \{\pm2, 4\},
	$$
	which contradicts $|\xi_4|\le |\xi|$. We conclude that Subcase A2' is empty.
	\vskip10pt
	\noindent\textbf{Case B.} $|\xi|\simeq |\xi_1|\simeq |\xi_2|\simeq |\xi_3|\gg |\xi_4|$. In this case, we interpolate between
\begin{equation}\label{eq:4kdvinterp1}
	\sup_{\xi_1,\xi_2,\xi_3} \int \frac{\jap{\xi}^{s+3/2+\epsilon}}{\prod_{j=1}^4\jap{\xi_j}^{s}}\fia d\xi_4
\end{equation} 
and
\begin{equation}\label{eq:4kdvinterp2}
	\sup_{\xi,\xi_4} \int \frac{\jap{\xi}^{s+1/2+\epsilon}}{\prod_{j=1}^4\jap{\xi_j}^{s}}\fia d\xi_1 d\xi_2.
\end{equation} 
First, notice that
$$
\frac{\jap{\xi}^{s+3/2+\epsilon}}{\prod_{j=1}^4\jap{\xi_j}^{s}}\lesssim |\xi_3|^{3/2+\epsilon - 3s},\quad  \frac{\jap{\xi}^{s+1/2+\epsilon}}{\prod_{j=1}^4\jap{\xi_j}^{s}}\lesssim |\xi|^{1/2+\epsilon - 3s}.
$$
For \eqref{eq:4kdvinterp1}, since $|\partial_{\xi_4}\Phi| \gtrsim |\xi_3|^2$, 
\begin{align*}
	\sup_{\xi_1,\xi_2,\xi_3} \int |\xi_3|^{3/2+\epsilon - 3s}\fia d\xi_4&\lesssim \sup_{\xi_1,\xi_2,\xi_3} \left(\int |\xi_3|^{3/2^++\epsilon - 3s}\fia d\xi_4\right)^{1^-}\\&\lesssim  \sup_{\xi_1,\xi_2,\xi_3} \left(\int |\xi_3|^{-1/2^++\epsilon - 3s}\fia d\Phi\right)^{1^-}\lesssim M^{1^-}.
\end{align*}
For \eqref{eq:4kdvinterp2}, write
$$
\Phi=\xi^3P(p_1,p_2,p_3,p_4),\quad p_j=\frac{\xi_j}{\xi},\ j=1,\dots, 4.
$$
As $p_4$ is fixed (near 0) and $p_3=1-p_1-p_2$, we have
$$
D^2P=-6\begin{bmatrix}
	p_1+p_3 & p_3 \\ p_3 & p_2+p_3
\end{bmatrix}.
$$
Therefore, since $p_1\simeq p_2\simeq -p_3\simeq 1$, $\det(D^2P)\simeq -1$. We now consider two cases:

\vskip5pt
\noindent 1. The nonstationary case $|\nabla P|\gtrsim 1$. Suppose, without loss of generality, that $|\partial_{p_1}P|\gtrsim 1$. Then
\begin{align*}
		\sup_{\xi,\xi_4} \int \frac{\jap{\xi}^{s+1/2+\epsilon}}{\prod_{j=1}^4\jap{\xi_j}^{s}}\fia d\xi_1 d\xi_2 &\lesssim 	\sup_{\xi,\xi_4} \int |\xi|^{5/2+\epsilon-3s}\fia dp_1dp_2 \\&\lesssim 	\sup_{\xi,\xi_4} \left(\int |\xi|^{5/2^++\epsilon-3s}\mathbbm{1}_{|\xi^3P-\alpha|<M} dp_1dp_2\right)^{1-}\\&\lesssim 	\sup_{\xi,\xi_4} \left(\int_{|p_2|\simeq 1} |\xi|^{5/2^++\epsilon-3s}\mathbbm{1}_{|\xi^3P-\alpha|<M} dPdp_2\right)^{1-}\lesssim M^{1^-}.
\end{align*}

\noindent 2. The stationary case $|\nabla P|\ll 1$. Thus we are near a nondegenerate critical point $z\in \R^2$. Applying Morse's Lemma, there exists a diffeomorphism $(p_1,p_2)\mapsto (q_1,q_2)$ such that
$$
P(p_1,p_2)=P(z)+q_1^2-q_2^2
$$
and, using Lemma \ref{lem:quadraticas},
\begin{align*}
	\sup_{\xi,\xi_4} \int \frac{\jap{\xi}^{s+1/2+\epsilon}}{\prod_{j=1}^4\jap{\xi_j}^{s}}\fia d\xi_1 d\xi_2 &\lesssim 	\sup_{\xi,\xi_4} \int |\xi|^{5/2+\epsilon-3s}\fia dp_1dp_2 \\&\lesssim 	\sup_{\xi,\xi_4} \left(\int |\xi|^{5/2^++\epsilon-3s}\mathbbm{1}_{|\xi^3(P(z)+q_1^2-q_2^2)-\alpha|<M} dq_1dq_2\right)^{1-}\lesssim M^{1^-}.
\end{align*}

\end{proof}

\begin{nb}
	In the above proof, the estimate follows from applying Lemma \ref{lem:interpol} for a specific choice of $A$. Since we aim to reach the scaling-critical regularity (that is, $\beta=1$), the integration of the indicator $\fia$ must give a full power of $M$. This is possible in two cases:
	\begin{itemize}
		\item If $\Phi$ is nonstationary, then the integration in $\Phi$ gives the correct power. This requires a single integration.
		\item If $\Phi$ is stationary and nondegenerate, we replace $\Phi$ with the corresponding quadratic form given by Morse's lemma. If this is done in two or more variables (determined by the number of integrals), the integration then yields the required estimate. However, if one has a single variable, the integration of the quadratic form gives $M^{1/2}$ (as mentioned in Remark \ref{nota:espaco}).
	\end{itemize}
In conclusion, with two integrations, the estimate holds as long as $\Phi$ does not have degenerate critical points, while in the case of a single integration, one must ensure that $\Phi$ is nonstationary. This is the motivation for the choices of pairings in both \eqref{eq:4kdvinter1}-\eqref{eq:4kdvinter2} and \eqref{eq:4kdvinterp1}-\eqref{eq:4kdvinterp2} made above.
\end{nb}

%
%\begin{nb}
%	In the above proof, the subtle reason why we need to use Lemma \ref{lem:interpol3} instead of Lemma \ref{lem:interpol} is the introduction the function $\psi(q_4)$ in the multiplier. Observe that $q_4$ is defined for $\xi,\xi_1$ fixed, and that its definition will change if we fix $\xi_2,\xi_3$ instead. Therefore, for both sides of the interpolation, these frequencies must remain fixed in order for $\psi(q_4)$ to be well-defined.
%\end{nb}

\bibliography{biblio}
\bibliographystyle{plain}

\begin{center}
	{\scshape Simão Correia}\\
	{\footnotesize
		Center for Mathematical Analysis, Geometry and Dynamical Systems,\\
		Department of Mathematics,\\
		Instituto Superior T\'ecnico, Universidade de Lisboa\\
		Av. Rovisco Pais, 1049-001 Lisboa, Portugal\\
		simao.f.correia@tecnico.ulisboa.pt
	}
	\vskip15pt
	{\scshape Filipe Oliveira}\\
	{\footnotesize
		Mathematics Department and CEMAPRE\\
		ISEG, Universidade de Lisboa,\\
		Rua do Quelhas 6, 1200-781 Lisboa, Portugal\\
		foliveira@iseg.ulisboa.pt
	}
	\vskip15pt
	
	{\scshape Jorge Drumond Silva}\\
	{\footnotesize
		Center for Mathematical Analysis, Geometry and Dynamical Systems,\\
		Department of Mathematics,\\
		Instituto Superior T\'ecnico, Universidade de Lisboa\\
		Av. Rovisco Pais, 1049-001 Lisboa, Portugal\\
		jsilva@math.tecnico.ulisboa.pt
	}

\end{center}
\end{document}